\RequirePackage{fix-cm}
\documentclass[smallextended]{svjour3}       
\smartqed  
\usepackage{graphicx}
\usepackage{amsmath,amssymb}
\usepackage{mathtools}
\usepackage{multirow}
\usepackage{color}

\makeatletter
  \def\@thmcountersep{.}
\makeatother

\newtheorem{theo}{Theorem}[section]
\newtheorem{cor}{Corollary}[section]
\newtheorem{lem}{Lemma}[section]
\newtheorem{rem}{Remark}[section]
\newtheorem{defi}{Definition}[section]

\def\qed{\hfill $\Box$}
\newcommand{\mc}{\mathcal}
\newcommand{\mb}{\mathbb}

\newcommand{\Honezero}{H^1_0(\Omega)}

\newcommand{\ra}{\rightarrow}
\newcommand{\f}{\frac}

\begin{document}

\title{A posteriori verification for the sign-change structure of solutions of elliptic partial differential equations
\thanks{This work was supported by JST CREST Grant Number JPMJCR14D4, and JSPS KAKENHI Grant Number 19K14601, and the Mizuho Foundation for the Promotion of Sciences.}
}


\author{Kazuaki Tanaka}


\institute{K. Tanaka \at
              Institute for Mathematical Science, Waseda University, 3-4-1, Okubo Shinjuku-ku, Tokyo 169-8555, Japan \\
              Tel.: +81-3-5286-2923
              \email{tanaka@ims.sci.waseda.ac.jp}           
}

\date{Received: date / Accepted: date}

\maketitle

\begin{abstract}
	This paper proposes a method for rigorously analyzing the sign-change structure of solutions of elliptic partial differential equations subject to one of the three types of homogeneous boundary conditions: Dirichlet, Neumann, and mixed.
	Given explicitly estimated error bounds between an exact solution $ u $ and a numerically computed approximate solution $ \hat{u} $,
	we evaluate the number of sign-changes of $ u $ (the number of nodal domains) and determine the location of zero level-sets of $ u $ (the location of the nodal line).
	We apply this method to the Dirichlet problem of the Allen–Cahn equation.
	The nodal line of solutions of this equation represents the interface between two coexisting phases.
\keywords{Numerical verification \and Sign-change structure \and Elliptic differentical equations \and Allen–Cahn equation \and Computer-assisted proof \and Verified numerical computation}
 \subclass{35J25 \and 35J61 \and 65N15}
\end{abstract}

\section{Introduction}\label{sec:1}
Numerical verification methods for partial differential equations have been developed in recent decades.
Such methods were first proposed in  \cite{nakao1988numerical,plum1991computer} and have been further developed by many researchers (see the recent survey book \cite{nakaoplumwatanabe2019numerical} and the references therein).
These approaches are also known as computer-assisted proofs, validated numerics, or verified numerical computations for partial differential equations
and have been applied to various problems, including some for which purely analytical methods have failed.
One such successful application is to the semilinear elliptic equation
\begin{align}
	\label{eq:main}
	-\Delta u(x)=f(u(x)), ~~x \in \Omega
\end{align}
with appropriate boundary conditions, 
where $\Delta $ is the Laplacian,
$\Omega \subset \mathbb{R}^{N}$~$(N=2,3,\cdots)$ is a bounded domain with a Lipschitz boundary,
and $ f: \mb{R} \ra \mb{R}$ is a nonlinear map
(see, for example, the numerical results in \cite{mckenna2009,plum2008,nakao2011numerical,mckenna2012computer,takayasu2013verified,tanaka2017sharp,nakaoplumwatanabe2019numerical}).
Further regularity assumptions for $\Omega$ and $f$ will be shown later for our setting.
Hereafter, $ H^k(\Omega) $ denotes the $k$-th order $ L^2 $ Sobolev space.
We define $ \Honezero := \{ u \in H^1(\Omega) : u = 0~\mbox{on} ~ \partial \Omega\} $, with the inner product $(u, v)_{H^1_0}:=(\nabla u, \nabla v)_{L^2}$ and norm $\| u \|_{H^1_0} := \sqrt{(u, u)_{\smash{H^1_0}}}$.

Numerical verification methods enable us to obtain an explicit ball containing exact solutions of \eqref{eq:main}.
More precisely, for a ``good'' numerical approximation $ \hat{u} \in \Honezero $, they enable us to prove the existence of an exact solution $ u \in H^1_0(\Omega) $ of \eqref{eq:main} that satisfies
\begin{align}
	\left\|u-\hat{u}\right\|_{H_{0}^{1}} \leq \rho \label{h10error}
\end{align}
with an explicit error bound $ \rho>0 $.
Additionally, under an appropriate condition, we can obtain an $ L^{\infty} $-estimation
\begin{align}
\left\|u-\hat{u}\right\|_{L^{\infty}} \leq \sigma \label{linferror}
\end{align}
with bound $ \sigma>0 $.
For instance, when $ u,\hat{u} \in H^2(\Omega) $, we can evaluate the $ L^{\infty} $-bound $ \sigma>0 $ by considering the embedding $ H^2(\Omega) \hookrightarrow L^{\infty}(\Omega) $;
details are discussed later in this section.
Thus, this approach has the advantage that quantitative information about the solutions of a target equation is provided accurately in a strict mathematical sense.
From the error estimates, we can identify the approximate shapes of solutions.
Despite these advantages, information about the  sign change of solutions is not guaranteed without additional considerations, irrespective of how small the error bound ($\rho$ or $\sigma$) is.
To be more precise, we introduce the following.
\begin{defi}
	\label{defi:nodaldomains}
	For $ u: \Omega \rightarrow \mathbb{R} $, the connected components of the open sets
	$$
	\{x \in \Omega~:~u(x)>0\}
	~~and~~
	\{x \in \Omega~:~u(x)<0\}
	$$
	are called {\bf the nodal domains} of $u$ and denoted by {\rm N.D.}$(u)$.
	In particular, $\{x \in \Omega~:~u(x)>0\}$ contains the positive nodal domains of $ u $ and is denoted by {\rm P.N.D.}$(u)$, and $\{x \in \Omega~:~u(x)<0\}$ contains the negative nodal domains of $ u $ and is denoted by {\rm N.N.D.}$(u)$.
	
	The zero level-set
	$$
	\{x \in \Omega~:~u(x)=0\}
	$$
	is called {\bf the nodal line} of $u$.
\end{defi}
According to the above definition, nodal lines do not contain the boundary of $ \Omega $; however,
we interpret zero-Dirichlet boundaries as parts of nodal lines when we apply this later to the Allen–Cahn equation (see Subsection \ref{subsec:ex}).

An essential problem is that $\#{\rm N.D.}(u)$ (the number of nodal domains) does not generally coincide with $\#{\rm N.D.}(\hat{u})$ (see Fig.~\ref{fig:nontrivial_nodalline}).
For example, when $ u $ is imposed on the homogeneous Dirichlet boundary conditions,
it is possible for $ u $ to be negative near the boundary $ \partial \Omega $ even when $ \hat{u} $ is positive in $ \Omega $.
In previous studies, we developed methods for verifying the positivity of solutions of \eqref{eq:maind} \cite{tanaka2015numerical,tanaka2017numerical,tanaka2017sharp,tanaka2020numerical}.
These methods succeeded in verifying the existence of positive solutions with precise error bounds by checking simple conditions, but determining the sign-change structure has been out of scope.
\begin{figure}[h]
	\begin{center}
		\includegraphics[height=40 mm]{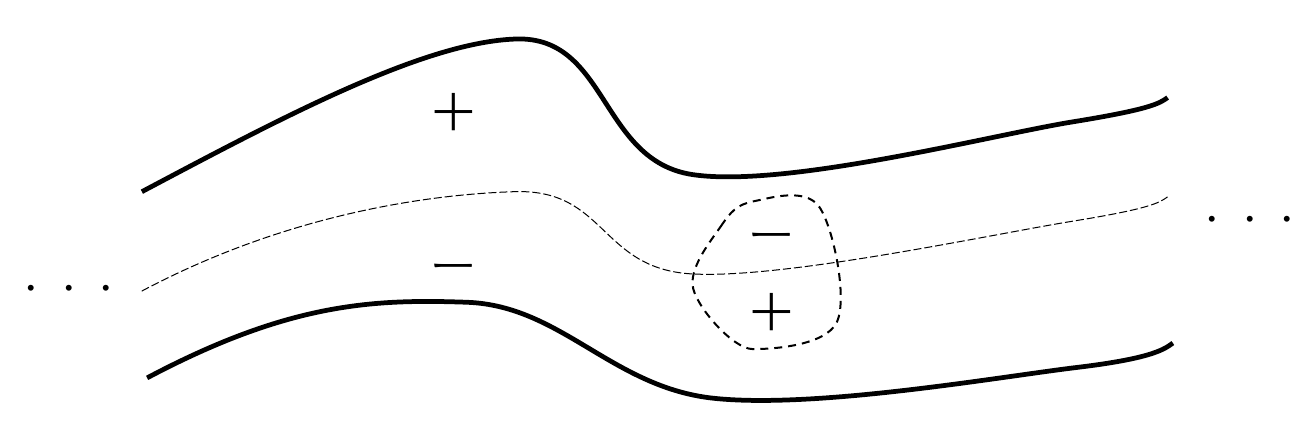}
	\end{center}
	\caption{Conceptual figure for the area in which $ (\hat{u}-\sigma)(\hat{u}+\sigma)<0 $ between the two solid lines.
	Nodal lines of $ u $ lay inside the area and do not exist outside.
	Regardless of how small $ \sigma>0 $ is, we cannot deny the possibility that there exist $ ( $small$ ) $ nodal domains in the area only from the error estimations $ \rho $ and/or $ \sigma $.
	If the nonexistence of nodal domains inside the area is proved, we can estimate $\#{\rm N.D.}(u)$ and determine the topology of the nodal lines $ ( $that is, how the lines intersect$ ) $.
	}
	\label{fig:nontrivial_nodalline}
\end{figure}

The main contribution of this paper is a proposed method for verifying the sign-change structure of solutions $ u $ of \eqref{eq:main} subject to one of the three types of homogeneous boundary value conditions---Dirichlet type, Neumann type and mixed type---while assuming the error estimations \eqref{h10error} and \eqref{linferror}.
If error bounds are sufficiently precise, our theorems can be applied to the case in which $ f $ is a subcritical polynomial 
\begin{align*}
f(t) = \lambda t + \sum_{i=2}^{n(<p^*)}a_{i}t|t|^{i-1},~~\lambda,~a_i \in \mb{R},~a_i \neq 0 \text{~~for~some~~} i,
\end{align*}
where
$p^*=\infty$ when $N=2$ and $p^*=(N+2)/(N-2)$ when $ N\geq3 $.	
They are also applicable to more general nonlinearities other than polynomials (see Theorems \ref{theo:main} and \ref{theo:mix}).
In the later sections, we discuss the applicability of our method to the Dirichlet problem
\begin{align}
\label{eq:maind}
\left\{\begin{array}{l l}
-\Delta u(x)=f(u(x)), &x \in \Omega,\\
u(x)=0,  &x \in \partial\Omega,\\
\end{array}\right.
\end{align}
the Neumann problem
\begin{align}
\label{eq:mainn}
\left\{\begin{array}{l l}
-\Delta u(x)=f(u(x)), &x \in \Omega,\\
\frac{\partial u}{\partial n}(x)=0,  &x \in \partial\Omega,\\
\end{array}\right.
\end{align}
and the mixed boundary value problem
\begin{align}
\label{eq:mainm}
\left\{\begin{array}{l l}
-\Delta u(x)=f(u(x)), &x \in \Omega,\\
u(x)=0,  &x \in \Gamma_D,\\
\frac{\partial u}{\partial n}(x)=0,  &x \in \Gamma_N,\\
\end{array}\right.
\end{align}
Here, $\Gamma_D $ is a relatively open subset of $ \partial \Omega $ and $\Gamma_N = \partial \Omega \backslash \overline{\Gamma_D} $.
We allow $\Gamma_D $ (or $\Gamma_N$) to be empty to unify \eqref{eq:mainn} (or \eqref{eq:maind}) with \eqref{eq:mainm}; otherwise,
we assume that both $\Gamma_D$ and $\Gamma_N$ are connected sets such that $\overline{\Gamma_D} \cap \overline{\Gamma_N} $ is an $(N-2)$-dimensional Lipschitz submanifold of $\partial \Omega$.
Therefore, when $\Omega \subset \mathbb{R}^2$ is simply connected, the intersection $\overline{\Gamma_D} \cap \overline{\Gamma_N} $ is composed of two points.
This assumption is not essential for our theory but can be weakened (see Remark \ref{rem:th31weekend}).

To our knowledge, $H^2$-regularity of solutions $u$ of the above elliptic problems (\eqref{eq:maind}, \eqref{eq:mainn}, or \eqref{eq:mainm}) is required for obtaining an $ L^{\infty} $-bound $ \sigma $ using existing methods.
We obtain an explicit bound for embedding $ H^2(\Omega) \hookrightarrow L^{\infty}(\Omega) $ using \cite[Theorem 1]{plum1992explicit}.
Moreover, we require an explicit bound $C>0$ that satisfies
\begin{align}
	\label{ineq:Hesse}
	\|v_{xx}\|_{L^2} \leq C\|\Delta v\|_{L^2}
\end{align}
for all $v \in H^2(\Omega)$ satisfying the boundary condition in \eqref{eq:maind}, \eqref{eq:mainn}, or \eqref{eq:mainm} in the trace sense, where $v_{xx}$ denotes the Hesse matrix of $v$.
When $\Omega$ is a polygonal domain, we have $\|v_{xx}\|_{L^2} = \|\Delta v\|_{L^2}$ for such $v$ and therefore can set $C=1$ (see \cite{grisvard2011elliptic}).
Combining the ideas from \cite[Section 4]{plum1992explicit} and \cite[Section 6.2.7]{nakaoplumwatanabe2019numerical} looks promising to prove inequality \eqref{ineq:Hesse} for more general domains, including in higher-dimensional cases.
The $ L^{\infty} $-bound $ \sigma $ can be derived by applying the embedding bound for $ H^2(\Omega) \hookrightarrow L^{\infty}(\Omega) $ and inequality \eqref{ineq:Hesse} to the error $u-\hat{u}$
when $\hat{u} \in H^2(\Omega)$ fulfills the same boundary condition imposed on $u$.
In this way, we obtain an $ L^{\infty} $-bound $ \sigma $ for the Dirichlet problem of the Allen-Cahn equation in Section \ref{subsec:ex}.
We believe that future methods can be developed to obtain $ L^{\infty} $-bounds without assuming $H^2$-regularity because weak solutions of these three problems always belong to $L^{\infty}(\Omega)$ when $ f $ is subcritical (see \cite[Corollary 6.6]{daners2009priori}).

We briefly explain some known facts about the $H^2$-regularity of solutions of the Poisson problem
\begin{align}
	\left\{\begin{array}{l l}
	-\Delta u=h &\mathrm{~in}\ \Omega,\\
	{\rm B.C.} &\mathrm{~on}\ \partial\Omega\\
	\end{array}\right. \label{eq:poisson}
\end{align}
given $h \in L^2(\Omega)$ and $\Omega \subset \mathbb{R}^2$ with corners,
where B.C. represents one of the three types of homogeneous boundary value conditions mentioned above.
For the zero-Dirichlet or zero-Neumann cases, i.e., when ${\rm B.C.}$ is replaced with $u=0$ or $\frac{\partial u}{\partial n}=0$,
solutions $u$ of \eqref{eq:poisson} have $H^2$-regularity if $\Omega$ is convex and has a piecewise $C^2$-boundary (see, for example, \cite{grisvard2011elliptic} and \cite[Subsection 5.3]{azegami2020boundary}).
For the mixed case, i.e., when ${\rm B.C.}$ is replaced with the boundary condition of \eqref{eq:mainm}, the opening angle $\alpha_{x_0}$ at a corner $x_0 \in \partial \Omega$ between $\Gamma_D $ and $\Gamma_N$ is essential for $H^2$-regularity.
If $\alpha_{x_0} \leq \frac{\pi}{2}$, solutions $u$ of \eqref{eq:poisson} have $H^2$-regularity around $x_0$ (see \cite[Subsection 5.3]{azegami2020boundary} for details).

The remainder of this paper is organized as follows.
In Section \ref{sec:diri}, we focus on the Dirichlet problem \eqref{eq:maind}, propose a method to estimate the number of nodal domains of solutions $ u $ and discuss the applicability of this method.
This section contains numerical applications of the method to the Allen--Cahn equation.
For several verified solutions, the number of nodal domains is estimated and then the locations of nodal lines are determined (see Subsection \ref{subsec:ex}).
Subsequently, in Section \ref{sec:othreboundary}, we extend our method to the other boundary value conditions: the Neumann type \eqref{eq:mainn} and mixed type \eqref{eq:mainm}.
\section{Verification for sign-change structure --- the Dirichlet case \eqref{eq:maind}}\label{sec:diri}
In this section, we limit our focus to the Dirichlet problem \eqref{eq:maind}.
Our scope will be extended in Section \ref{sec:othreboundary}.	

We begin by introducing required notation.
We denote $ V=\Honezero $ and $ V^*=$ (the topological dual of $ V $).
For two Banach spaces $X$ and $Y$, the set of bounded linear operators from $X$ to $Y$ is denoted by ${\mc L}(X, Y)$ with the usual supremum norm $\| T \|_{{\mc L}(X, Y)} := \sup\{ \| T u \|_{Y} / \| u \|_{X} : {u \in X \setminus \{0\}}\}  $ for $T \in {\mc L}(X, Y)$.
A norm bound for the embedding $V\hookrightarrow L^{p+1}\left(\Omega\right)$ is denoted by $C_{p+1}(=C_{p+1}(\Omega))$; that is, $C_{p+1}$ is a positive number that satisfies
\begin{align}
\label{embedding}
\left\|u\right\|_{L^{p+1}(\Omega)}\leq C_{p+1}\left\|u\right\|_{V}~~~{\rm for~all}~u\in V,
\end{align}
where $p\in [1,\infty)$ when $N=2$ and $p\in [1,p^*]$ when $ N\geq3 $.
If no confusion arises, we use the notation $ C_{p+1} $ to represent the embedding constant on the entire domain $ \Omega $,
whereas, in some parts of this paper, we must consider an embedding constant on some subdomain $ \Omega'\subset\Omega $.
This is denoted by $ C_{p+1}(\Omega') $ to avoid confusion.
Moreover, $ \lambda_1(\Omega) $ denotes the first eigenvalue of $ -\Delta $ imposed on the homogeneous Dirichlet boundary condition.
This is characterized by
\begin{align}
\label{eq:defi-diri-eigenvalue}
\lambda_{1}(\Omega) = \inf_{v\in V^\backslash{\{0\}}} \frac{\|v\|_{V}^2}{\|v\|_{L^2}^2}.
\end{align}		
Note that, when domains $\Omega_1, \Omega_2 \subset \mathbb{R}^N$ satisfy $\Omega_1 \subset \Omega_2 $, $C_{p+1}(\Omega_2)$ can be used as a bound $C_{p+1}(\Omega_1)$ by considering the zero-extension outside $\Omega_1$ to $\Omega_2$ for $u \in H^1_0(\Omega_1) \subset H^1_0(\Omega_2)$.
In the same way, we confirm $\lambda_1(\Omega_1)\geq \lambda_1(\Omega_2)$.

Throughout this paper, we assume that $ f $ is a $ C^1 $ function that satisfies
\begin{align*}
&|f(t)| \leq a_0 |t|^p + b_0 \text{~~~for~all~} t \in \mb{R},\\
&|f'(t)| \leq a_1 |t|^{p-1} + b_1 \text{~~~for~all~} t \in \mb{R}
\end{align*}
for some $ a_0,a_1,b_0,b_1\geq 0 $ and $ p<p^* $.
We define the operator $ F $ by
\begin{align*}
F : \left\{\begin{array}{ccc}{u(\cdot)} & {\mapsto} & {f(u(\cdot))}, \\
{V} & {\rightarrow} & {V^{*}}.\end{array}\right.
\end{align*}
Moreover, we define another operator $ \mathcal{F} : V \rightarrow V^{*}$ by $\mathcal{F}(u) :=-\Delta u-F(u) $, which is characterized by
\begin{align}
\left<\mathcal F(u),v\right> = \left(\nabla u,\nabla v\right)_{L^2} - \left<F(u),v\right>  \text{~~~for~all~} u,v \in V,
\end{align}
where $\left<F(u),v\right> = \int_{\Omega} f(u(x)) v(x) dx$.
The Fr\'echet derivatives of $ F $ and $ \mc{F} $ at $ \varphi \in V $, denoted by $ {F'_{\varphi}} $ and $ {\mc F'_{\varphi}} $, respectively, are given by
\begin{align}
&\langle F'_{\varphi}u,v\rangle = \int_{\Omega} f'(\varphi(x))u(x) v(x) dx \text{~~~for~all~} u,v \in V,\\
&\langle \mathcal F'_{\varphi}u,v \rangle = \left(\nabla u,\nabla v\right)_{L^2} - \langle F'_{\varphi}u,v \rangle  \text{~~~for~all~} u,v \in V. \label{def:derivativecalf}
\end{align}
Under the notation and assumptions, we look for solutions $ u \in V $ of 
\begin{align}
\label{main:fpro}
\mathcal{F}(u)=0,
\end{align}
which corresponds to the weak form of \eqref{eq:maind}.
We call this the {\it D-problem} to prevent confusion with the other boundary value problems to be discussed in Section \ref{sec:othreboundary}.
Recall that the weak solution $ u \in V $ of the D-problem is in $L^{\infty}(\Omega)$; see \cite[Corollary 6.6]{daners2009priori}.
We assume that some numerical verification method succeeds in proving the existence of a solution $u \in V \cap L^{\infty}(\Omega)$ of \eqref{main:fpro} in the intersection of
\begin{align}
&\overline{B}(\hat{u},\rho,\|\cdot\|_V) := \left\{v\in V : \left\|v-\hat{u}\right\|_{V}\leq\rho\right\} \label{eq:h10ball},\\
&\overline{B}(\hat{u},\sigma,\|\cdot\|_{L^{\infty}}) := \left\{v\in L^{\infty}(\Omega) : \left\|v-\hat{u}\right\|_{L^{\infty}}\leq\sigma\right\} \label{eq:linfball}
\end{align}
given $ \hat{u} \in V \cap L^{\infty}(\Omega) $ and $ \rho,\sigma >0 $.
Although the regularity assumption for $ \hat{u} $ (to be in $ \in V \cap L^{\infty}(\Omega) $) is theoretically sufficient to obtain the error bounds \eqref{eq:h10ball} and \eqref{eq:linfball},
we further assume that $ \hat{u} $ is continuous or piecewise continuous.
This assumption impairs little of the flexibility of actual numerical verification methods.
Indeed, past verification was implemented with such approximate solutions $ \hat{u} $;
again, see \cite{mckenna2009,plum2008,nakao2011numerical,mckenna2012computer,takayasu2013verified,tanaka2017sharp,nakaoplumwatanabe2019numerical}.
Then, we use the following notation:
\begin{itemize}
	\item[] $\Omega_{+}:=\{x\in\Omega : \hat{u}-\sigma>0\}$ where $ u>0 $ therein;
	\item[] $\Omega_{-}:=\{x\in\Omega : \hat{u}+\sigma<0\}$ where $ u<0 $ therein;
	\item[] $\Omega_{0}:= \Omega \backslash(\Omega_{+}\cup\Omega_{-})$.
\end{itemize}
The subset $ \Omega_0 $ approximates the nodal line of $ u $, and therefore the location of $ \Omega_0 $ is essential for determining the topology of the nodal line.
In practice, $\Omega_{+}$ and $\Omega_{-}$ are set to a subset of $\{x\in\Omega : \hat{u}-\sigma>0\}$ and $\{x\in\Omega : \hat{u}+\sigma<0\}$), respectively, then $ \Omega_{0} $ is defined as above.
This generalization can be applied directly to our theory.
We assume the following geometric properties:
\begin{itemize}
	\item[] $\Omega_{+}$, $\Omega_{-}$, and $\mathring\Omega_{0}$ are Lipschitz subdomains composed of a finite number of connected components, where $\mathring\Omega_{0}$ denotes the interior of $\Omega_{0}$.
	\item[] $ \Omega_{0} $ is not empty, and satisfies $\Omega_{0} = \overline{\mathring\Omega_{0}} \cap \Omega$.
	\item[] $ \sigma $ is small so that $ \Omega_{0} \neq \Omega $.
\end{itemize}
\subsection{Main theorem}\label{subsec:maintheo}
The following lemma plays an essential role in our main result.
\begin{lem}\label{lem:positive}
	Let $ f $ satisfy
	\begin{align}
	\label{f:lem1}
	tf(t)\leq \lambda t^2 + \displaystyle \sum_{i=1}^{n}a_{i}|t|^{p_{i}+1} \text{~~~for~all~}t \in \mathbb{R}
	\end{align}
	for some $ \lambda < \lambda_1(\Omega) $, nonnegative coefficients $ a_1,a_2,\cdots,a_n $, and subcritical exponents $ p_1,p_2,\cdots,p_n \in (1,p^*)$.
	If a solution $ u \in V $ of the D-problem \eqref{main:fpro} satisfies the inequality
	\begin{align}
	\label{cond:lem1}
	\displaystyle \sum_{i=1}^{n}a_i C_{p_{i}+1}^{2} \left\|u\right\|_{L^{p_{i}+1}}^{p_{i}-1}<1 - \f{\lambda}{\lambda_1(\Omega)},
	\end{align}		
	then $ u $ is the trivial solution $ u\equiv 0 $, where $ C_{p_{i}+1}=C_{p_{i}+1}(\Omega) $.
\end{lem}
\begin{rem}
	The left-hand side of \eqref{cond:lem1} converges to zero as $\left\|u\right\|_{L^{p_{i}+1}} \downarrow 0$.
	Therefore, if the solution $ u $ of \eqref{main:fpro} is sufficiently small to satisfy \eqref{cond:lem1}, then $ u $ must vanish.
\end{rem}
\begin{rem}
	The inequality \eqref{f:lem1} can be reduced to a combination of the following inequalities:
	\begin{align*}
	&f(t)\leq \lambda t + \displaystyle \sum_{i=1}^{n}a_{i}t^{p_{i}} \text{~~~for~all~}t \geq 0,\\
	&-f(-t)\leq \lambda t + \displaystyle \sum_{i=1}^{n}a_{i}t^{p_{i}} \text{~~~for~all~}t \geq 0.
	\end{align*}
	Therefore, 	the polynomial
	$ f(t) = \lambda t + \sum_{i=2}^{n(<p^*)}a_{i}t|t|^{i-1} $ 
	with $ \lambda < \lambda_1(\Omega) $ and $a_i \in \mb{R} $ obviously satisfies the required inequality \eqref{f:lem1}.
	Indeed, for the set of subscripts $ \Lambda_{+} $ for which $ a_i \geq 0~(i \in \Lambda_{+})$ and $ a_i < 0$ $($otherwise$)$,
	we have $f(t)\leq \lambda t + \sum_{i \in \Lambda_{+}}a_{i}t^{i} $ and $-f(-t)\leq \lambda t + \sum_{i \in \Lambda_{+}}a_{i}t^{i} \text{~for~all~}t\geq 0 $.
\end{rem}
\subsubsection*{Proof of Lemma \ref{lem:positive}}
We prove that $ \|u\|_V = 0 $.
Because $u$ satisfies
\begin{align*}
\left(\nabla u,\nabla v\right)_{L^2}=\left<F(u),v\right>{\rm~~~for~all~}v\in V,
\end{align*}
by fixing $ v=u $,
we have
\begin{align}
\label{ineq:prooflemma}
\left\|u \right\|_{V}^{2} 
\leq &\displaystyle \int_{\Omega}\left\{\lambda\left(u (x)\right)^{2}+\sum_{i=1}^{n}a_{i} |u (x)|^{p_{i}+1}\right\}dx  \nonumber \\
= &\displaystyle \lambda\left\|u \right\|_{L^{2}}^{2}+\sum_{i=1}^{n}a_{i}\left\|u \right\|_{L^{p_{i}+1}}^{p_{i}+1}\nonumber \\
\leq &\left\{ \displaystyle \frac{\lambda}{\lambda_{1}(\Omega)} +\sum_{i=1}^{n}a_{i}C_{p_{i}+1}^{2} \left\|u \right\|_{L^{p_{i}+1}}^{p_{i}-1} \right\} \left\|u \right\|_{V}^{2}.
\end{align}
Therefore, \eqref{cond:lem1} ensures that $ \|u \|_V = 0 $.
\qed\\

For two sets $A, B \subset \mathbb{R}^N $, we denote by $\#\text{\rm C.C.}(B;A)$ the number of connected components $B^i$ $(i=1,2,\cdots)$ of $ B $ such that $A\cap B^i \neq \emptyset$.
We simply write $\#\text{\rm C.C.}(A)$ = $\#\text{\rm C.C.}(A;A)$, the number of all the connected components of $A$.
Before describing the main theorem (Theorem \ref{theo:main}), we prepare the following lemma.
\begin{lem}\label{lem:cc}
Let $A, B \subset \mathbb{R}^N $ be composed of a finite number of connected components $A^i \subset A$ and $B^i \subset B$ $(i=1,2,\cdots)$.
If $A \subset B$,
each connected component $B^i$ of $B$ such that $B^i \cap A \neq \emptyset$ contains a connected component of $A$,
and thus
\begin{align}
	\#\text{\rm C.C.}(B;A) \leq \#\text{\rm C.C.}(A).
\end{align}
\end{lem}

\proof
When $B^i \cap A \neq \emptyset$, there exists a connected component $A^j$ of $A$ such that $B^i \cap A^j \neq \emptyset$.
Let $x \in B^i \cap A^j$.
One confirms that $B^i$ is the maximal connected subset of $B$ that contains $x$, and $A^j$ is a subset in $B$ that contains $x$.
Thus, $A^j \subset B^i$.
\qed\\

On the basis of Lemmas \ref{lem:positive} and \ref{lem:cc}, the following theorem evaluates the number of nodal domains of $ u $ from the inclusions \eqref{eq:h10ball} and \eqref{eq:linfball} for $ \hat{u} $.
\begin{theo}\label{theo:main}
	Let $ f $ satisfy \eqref{f:lem1} for some $ \lambda < \lambda_1(\mathring\Omega_0) $.
	We denote $ C_{p_{i}+1} = C_{p_{i}+1}(\Omega) $.
	If	
	\begin{align}
	\label{cond:theo1}
	\displaystyle \sum_{i=1}^{n}a_i C_{p_{i}+1}(\mathring\Omega_0)^{2}\left( \left\|\hat{u}\right\|_{L^{p_{i}+1}(\mathring\Omega_{0})}+C_{p_{i}+1}\rho\right)^{p_{i}-1}<1 - \f{\lambda}{\lambda_1(\mathring\Omega_0)},
	\end{align}		
	then a solution $u \in V \cap L^{\infty}(\Omega)$ of the D-problem \eqref{main:fpro} existing in the intersection of the balls \eqref{eq:h10ball} and \eqref{eq:linfball} satisfies
	\begin{align}
	&\#\text{\rm C.C.}(\Omega_{+}\cup\Omega_0; \Omega_{+}) \leq \#\text{\rm P.N.D.}(u) \leq \#\text{\rm C.C.}(\Omega_{+}), \label{ineq:cc1}\\
	&\#\text{\rm C.C.}(\Omega_{-}\cup\Omega_0; \Omega_{-}) \leq \#\text{\rm N.N.D.}(u) \leq \#\text{\rm C.C.}(\Omega_{-}).\label{ineq:cc2}
	\end{align}
	Note that if $ \mathring\Omega_0 $ is disconnected, \eqref{cond:theo1} is understood as the set of inequalities for all connected components $\mathring\Omega_0^j$ $ ( $$j=1,2,\cdots$$ ) $ of $ \mathring\Omega_0 $.
\end{theo}
\begin{rem}
	The formula inside the parentheses in \eqref{cond:theo1} converges to $0$ as $ \rho \downarrow 0 $ and $ |\Omega_{0}| \downarrow 0 $, which is equivalent to $ \sigma\downarrow0 $ when $ \hat{u} $ is continuous.
	Therefore, if verification succeeds for a continuous approximation $\hat{u}$ with sufficient accuracy,
	the number of nodal domains of $ u $ can be evaluated using Theorem $\ref{theo:main}$.
\end{rem}
\begin{rem}
	The connected components on either side of the inequalities \eqref{ineq:cc1} and \eqref{ineq:cc2} can be determined only from the information on the approximation $ \hat{u} $ and the $ L^{\infty} $-error $ \sigma $ as in \eqref{eq:linfball}; see the definitions of $ \Omega_{+} $, $ \Omega_{-} $, and $ \Omega_{0} $ $  $located just before Lemma $\ref{lem:positive}  $.
\end{rem}
\begin{rem}
	Explicitly estimating a lower bound for $ \lambda_1(\mathring\Omega_{0}) $ and upper bounds for $ C_{p+1} (=C_{p+1}(\Omega))$ and $ C_{p+1}(\mathring\Omega_0)$ is essential for Theorem {\rm \ref{theo:main}}.
	This topic is discussed in Appendix \ref{sec:constants}.
\end{rem}
\subsubsection*{Proof of Theorem \ref{theo:main}}
We first prove that there is no nodal domain of $ u $ in $ \mathring\Omega_{0} $.
To achieve this, we confirm that if $ u|_{\Omega'} $ (the restriction of $ u $ over $\Omega'$) can be regarded as a solution of the D-problem \eqref{main:fpro} for some subdomain $ \Omega' \subset \mathring\Omega_0 $ with the notational replacement $ \Omega \rightarrow \Omega' $, then $ u|_{\Omega'} $ should be a trivial solution that satisfies $ u|_{\Omega'}\equiv 0 $.

Suppose that there exists such a subdomain $ \Omega' $ so that $ u|_{\Omega'} \in H^1_0(\Omega')~(\subset V) $ is a solution of the D-problem \eqref{main:fpro} with the replacement $ \Omega \rightarrow \Omega' $.
We express $u\in V$ as $ u=\hat{u}+ \rho\omega$, where $\omega\in V$ satisfies $\left\|\omega\right\|_{V}\leq 1$.
This ensures that, for $p\in (1,p^*)$,
\begin{align}
\left\|u\right\|_{L^{p+1}(\Omega')}
\leq
\left\|\hat{u}\right\|_{L^{p+1}(\Omega')}+ C_{p+1}\rho	
\end{align}	
because $\left\|\omega \right\|_{L^{p+1}(\Omega')}\leq \left\|\omega \right\|_{L^{p+1}(\Omega)} \leq C_{p+1}\left\|\omega \right\|_{V}\leq C_{p+1}$.
Therefore,  it readily follows from $ \left\|\hat{u} \right\|_{L^{p+1}(\Omega')} \leq \left\|\hat{u} \right\|_{L^{p+1}(\mathring\Omega_0)} $ that
\begin{align}
\left\|u \right\|_{L^{p+1}(\Omega')}
\leq
\left\|\hat{u} \right\|_{L^{p+1}(\mathring\Omega_0)}+ C_{p+1}\rho.
\label{u:eval}
\end{align}	
Therefore, \eqref{cond:theo1} and \eqref{u:eval} ensure that
\begin{align*}
	\sum_{i=1}^{n}a_i C_{p_{i}+1}(\mathring\Omega_0)^{2} \left\| u|_{\Omega'}\right\|_{L^{p_{i}+1}(\Omega')}^{p_{i}-1}
	< 1 - \f{\lambda}{\lambda_1(\mathring\Omega_0)}
	\leq 1 - \f{\lambda}{\lambda_1(\Omega')},
\end{align*}
where $\left\|u \right\|_{L^{p_i+1}(\Omega')} = \left\|u|_{\Omega'} \right\|_{L^{p_i+1}(\Omega')}$ and $\lambda_1(\Omega')\geq \lambda_1(\mathring\Omega_0)$.
Because $C_{p+1}(\mathring\Omega_0) $ can be used as a bound $C_{p+1}(\Omega')$,
it follows from Lemma \ref{lem:positive} that $ u|_{\Omega'}\equiv 0 $.
Thus, there is no nodal domain in $ \mathring\Omega_0 $.

In the following, we evaluate the number of nodal domains of $u$.
Let us write $\Omega^{*}_{+}=\{x\in \Omega:u(x)>0\}$, so that $\text{\rm C.C.}(\Omega^{*}_{+})=\text{\rm P.N.D.}(u)$.
Because there is no positive nodal domain in the interior of $ \Omega_0 \cup \Omega_{-} $, we have
\begin{align}
	&\text{\rm C.C.}(\Omega^{*}_{+};\Omega_{+}) = \text{\rm C.C.}(\Omega^{*}_{+}) = \text{\rm P.N.D.}(u).
\end{align}
Because $\Omega_{+}\subset \Omega^{*}_{+}$, it follows from Lemma \ref{lem:cc} that
\begin{align}
	\#\text{\rm P.N.D.}(u) = \#\text{\rm C.C.}(\Omega^{*}_{+};\Omega_{+}) \leq \#\text{\rm C.C.}(\Omega_{+}). \label{ineq:ccright}
\end{align}
Thus, the right inequality in \eqref{ineq:cc1} is proved.
Besides, Lemma \ref{lem:cc} indicates from the inclusion $\Omega^{*}_{+} \subset \Omega_{+}\cup\Omega_0$ that 
\begin{align}
	\#\text{\rm C.C.}(\Omega_{+}\cup\Omega_0;\Omega^{*}_{+}) \leq \#\text{\rm C.C.}(\Omega^{*}_{+})~(= \#\text{\rm P.N.D.}(u)). \label{ineq:ccleft}
\end{align}
Again, from the inclusion $\Omega_{+}\subset \Omega^{*}_{+}$, we see that
$\#\text{\rm C.C.}(\Omega_{+}\cup\Omega_0; \Omega_{+}) \leq \#\text{\rm C.C.}(\Omega_{+}\cup\Omega_0;\Omega^{*}_{+})$.
This ensures the left inequality in \eqref{ineq:cc1}.

Inequality \eqref{ineq:cc2} is ensured in the same way with the notational replacements $\Omega_{+} \rightarrow \Omega_{-}$ and $\Omega^{*}_{+} \rightarrow \Omega^{*}_{-} := \{x\in \Omega:u(x)<0\}$.
\qed

\subsection{Further discussion on the main theorem}
In this subsection, we provide some remarks about Theorem \ref{theo:main}.
\subsubsection{Inequality \eqref{f:lem1} can be weakened}
Assuming the $ L^{\infty} $-error estimation \eqref{linferror} (or \eqref{eq:linfball}), we ensure that the range of $ u $ is taken over $ [\min\{\hat{u}\}-\sigma,\max\{\hat{u}\}+\sigma] $.
Therefore, the condition \eqref{f:lem1} imposed on $ f $ is replaceable with
\begin{align}
	\label{f:lem1_rep}
	tf(t)\leq \lambda t^2 + \displaystyle \sum_{i=1}^{n}a_{i}|t|^{p_{i}+1} \text{~~~for~all~} t \in [\min\{\hat{u}\}-\sigma,\max\{\hat{u}\}+\sigma]
\end{align}
because \eqref{ineq:prooflemma} is confirmed in the same manner when the $ L^{\infty} $-error $ \sigma $ is explicitly estimated. 
\subsubsection{When assuming only an $ L^{\infty} $-error}
Given $ \sigma $ satisfying \eqref{linferror}, $ u $ can be written as $ u=\hat{u} + \sigma \omega $ with $ \omega \in L^{\infty}(\Omega) $ satisfying $ \|\omega\|_{L^{\infty}} \leq 1$.
Therefore, applying the inequality
\begin{align}
	\left\|u\right\|_{L^{p+1}(\Omega')}
	\leq
	\left\|\hat{u}\right\|_{L^{p+1}(\mathring\Omega_0)}+ \sigma |\Omega_0|^{\frac{1}{p+1}}
\end{align}
instead of \eqref{u:eval}, we have the following similar theorem without assuming an $ H^1_0 $-error $ \rho $ but only an $ L^{\infty} $-error $ \sigma $.
\begin{theo}\label{theo:main2}
	Let $ f $ satisfy \eqref{f:lem1} for some $ \lambda < \lambda_1(\mathring\Omega_0) $.
	If	
	\begin{align}
		\displaystyle \sum_{i=1}^{n}a_i C_{p_{i}+1}(\mathring\Omega_0)^{2}\left(\left\|\hat{u}\right\|_{L^{p_{i}+1}(\mathring\Omega_0)}+ \sigma |\Omega_0|^{\frac{1}{p_{i}+1}}\right)^{p_{i}-1}<1 - \f{\lambda}{\lambda_1(\mathring\Omega_0)},
	\end{align}		
	then a solution $u \in V \cap L^{\infty}(\Omega)$ of the D-problem \eqref{main:fpro} existing in the ball \eqref{eq:linfball} satisfies \eqref{ineq:cc1} and \eqref{ineq:cc2}.
\end{theo}
Note that almost all existing verification methods for the partial differential equation \eqref{eq:main} estimate an $ L^{\infty} $-error $ \sigma $ after deriving an $ H^1_0 $-error $ \rho $, as described in Subsection \ref{subsec:ex} (see, e.g., \cite{nakaoplumwatanabe2019numerical}).
However, if $ \sigma $ is obtained directly without computing $ \rho $, Theorem \ref{theo:main2} becomes useful.
\subsubsection{Sufficient conditions for \eqref{cond:theo1}}
Because $C_{p_{i}+1}(\Omega)$ can be regarded as a bound $ C_{p_{i}+1}(\mathring\Omega_0) $,
the following simplified inequality is sufficient for \eqref{cond:theo1}.
\begin{align*}
\displaystyle \sum_{i=1}^{n}a_i C_{p_{i}+1}^{2}\left( \left\|\hat{u}\right\|_{L^{p_{i}+1}(\mathring\Omega_{0})}+C_{p_{i}+1}\rho\right)^{p_{i}-1}<1 - \f{\lambda}{\lambda_1(\mathring\Omega_0)}.
\end{align*}
If we have $ \lambda < \lambda_1(\Omega) $, this is further reduced to
\begin{align*}
\displaystyle \sum_{i=1}^{n}a_i C_{p_{i}+1}^{2}\left( \left\|\hat{u}\right\|_{L^{p_{i}+1}(\mathring\Omega_{0})}+C_{p_{i}+1}\rho\right)^{p_{i}-1}<1 - \f{\lambda}{\lambda_1(\Omega)}
\end{align*}
because $\lambda_1(\mathring\Omega_0)\geq \lambda_1(\Omega)$.
Generally, the shape of $ \Omega_{0} $ tends to be more complicated than $ \Omega $, which makes the evaluation of $ C_{p_{i}+1}(\mathring\Omega_0) $ and/or $ \lambda_1(\mathring\Omega_0) $ difficult.
The above sufficient inequalities can be useful in such cases.
\subsubsection{Application to specific nonlinearities}
We apply Theorem \ref{theo:main} to two specific problems in which we are interested.
The first problem is \eqref{eq:maind} with the nonlinearity $ f(t) = \lambda t + t|t|^{p-1} $, $p\in (1,p^*)$.
Adapting Theorem \ref{theo:main} to this case, we have the following.
\begin{cor}
	\label{cor:emden}
	Let $ f(t)= \lambda t + t|t|^{p-1} $, with $ p \in (1,p^*) $.
	If
	\begin{align}
	C_{p+1}(\mathring\Omega_0)^{2}\left(\left\|\hat{u}\right\|_{L^{p+1}(\mathring\Omega_0)}+C_{p+1}\rho\right)^{p-1}<1 - \f{\lambda}{\lambda_1(\mathring\Omega_0)},
	\end{align}
	then a solution $u \in V$ of the D-problem \eqref{main:fpro} in the intersection of the balls \eqref{eq:h10ball} and \eqref{eq:linfball} satisfies \eqref{ineq:cc1} and \eqref{ineq:cc2}.
\end{cor}

The second problem is the case in which $f(t)=\varepsilon^{-2}(t-t^3)$ $(\varepsilon>0)$.
We only consider the case where $ \varepsilon^{-2} \geq \lambda_1(\Omega) $,
because there is no solution of the D-problem \eqref{main:fpro} other than the trivial solution $ u\equiv 0 $ when $\varepsilon^{-2} < \lambda_1(\Omega)$.
Indeed, no positive solution is admitted when $\varepsilon^{-2} < \lambda_1(\Omega)$.
This can be confirmed by multiplying $-\Delta u = \varepsilon^{-2}(u-u^3)$ with the first eigenfunction of $-\Delta$ and integrating both sides.
For a sign-changing solution $u$, let $\Omega'$ be a positive nodal domain of $u$.
Note that $-u$ is also a solution of \eqref{main:fpro}, and therefore, considering only positive nodal domains is sufficient.
The restricted function $u_{\Omega'}$ is a solution of a zero-Dirichlet problem restricted on $\Omega'$ and $ \lambda_1(\Omega) \leq \lambda_1(\Omega')$.
Thus, if $\varepsilon^{-2} < \lambda_1(\Omega) (\leq \lambda_1(\Omega'))$, $u$ is the trivial solution.

Because
\begin{align}
	tf(t)\leq \varepsilon^{-2} t^2 \text{~~~for~all~}t \in \mathbb{R},
\end{align}
applying Theorem \ref{theo:main} to the nonlinearity gives us the following.
\begin{cor}\label{coro:main}
	Let $f(t)=\varepsilon^{-2}(t-t^3)$, with $ \varepsilon^{-2} \geq \lambda_1(\Omega) $.
	If	
	\begin{align}
	\label{cond:coro1}
	\varepsilon^{-2} < \lambda_1(\mathring\Omega_0),
	\end{align}
	then a solution $u \in V$ of the D-problem \eqref{main:fpro} in the intersection of the balls \eqref{eq:h10ball} and \eqref{eq:linfball} satisfies \eqref{ineq:cc1} and \eqref{ineq:cc2}.
\end{cor}
In the next subsection, Corollary \ref{coro:main} is applied to an important problem.
\subsection{Numerical example}\label{subsec:ex}
In this subsection,
we consider the stationary problem of the Allen--Cahn equation:
\begin{align}
\label{eq:allen}
\left\{\begin{array}{l l}
-\Delta u(x)=\varepsilon^{-2}(u(x)-u(x)^3), &x \in \Omega,\\
u(x)=0,  &x \in \partial\Omega\\
\end{array}\right.
\end{align}
for which Corollary \ref{coro:main} can be used.
The Allen–Cahn equation was originally proposed as a simplified model for the phase separation process \cite{allen1979microscopic}.
Because the nodal line of solutions of this equation represents the interface between two coexisting phases, determining its location is important for the problem.

We demonstrated the applicability of our theory to the problem on square $ \Omega =(0,1)^2 $.
All computations were implemented on a computer with 2.20 GHz Intel Xeon E7-4830 CPUs $\times$ 4, 2 TB RAM, and CentOS 7 using MATLAB 2019b with GCC version 6.3.0. 
All rounding errors were strictly estimated using the toolboxes INTLAB version 11 \cite{rump1999book} and kv Library version 0.4.48 \cite{kashiwagikv}.
Therefore, the accuracy of all results was guaranteed mathematically.
We constructed approximate solutions of \eqref{main:fpro} for the domain via a Legendre polynomial basis.
Specifically, we define a finite-dimensional subspace $ V_M~(\subset V) $ as the tensor product $V_M = \text{span\,}\{\phi_{1},\phi_{2},\cdots,\phi_{M}\} \otimes \text{span\,}\{\phi_{1},\phi_{2},\cdots,\phi_{M}\} $, where each $\phi_{n}$ ($ n=1,2,3,\cdots $) is defined as
\begin{align}
	&\displaystyle \phi_{n}(x)=\frac{1}{n(n+1)}x(1-x)\frac{dQ_{n}}{dx}(x)\nonumber \\
	&\text{~~with~~}
	Q_{n}=\displaystyle \frac{(-1)^{n}}{n!}\left(\frac{d}{dx}\right)^{n}x^{n}(1-x)^{n},~~n=1,2,3,\cdots. 
\end{align}
For a fixed integer $ M_u \geq 1 $, we construct $\hat{u}$ in $ V_{M_u} $ as
\begin{align}
\label{eq:uhat}
\displaystyle \hat{u}(x,y)=\sum_{i=1}^{M_u}\sum_{j=1}^{M_u}u_{i,j}\phi_{i}(x)\phi_{j}(y),~~u_{i,j}\in \mathbb{R}.
\end{align}
Note that our method does not limit the basis functions that constitute approximate solutions,
but can be applied to many bases other than the Legendre polynomial basis, such as the finite element and Fourier bases.

In actual computations to obtain $ H^1_0 $-errors $ \rho $ using the methods proposed in \cite{plum2008,tanaka2014verified}, verification was implemented on the solution space $ V $ with the generalized inner product and norm
\begin{align}
 	\label{tau_norm}
	\left(u,v\right)_{\tau}=\left(\nabla u,\nabla v\right)_{L^{2}}+\tau\left(u,v\right)_{L^{2}},~~~\left\| u \right\|_{\tau}=\sqrt{\left(u,u\right)_{\tau}},
\end{align}
where $\tau$ is a nonnegative number chosen as
\begin{align}
	\label{select:tau}
	\tau>-f'(\hat{u}(x)) = \varepsilon^{-2} (-1 + 3 \hat{u}^2) \text{~~~for~all~} x\in\Omega.
\end{align}
However, because the norm $\left\|\cdot\right\|_{\tau}$ monotonically increases with $\tau$, the usual norm $\left\|\cdot\right\|_{V} (= \left\|\nabla\cdot\right\|_{L^{2}})$ is bounded by $\left\|\cdot\right\|_{\tau}$ for any $\tau \geq 0$.
Therefore, we can use the error bound $ \|u-\hat{u}\|_{\tau} $ as the error bound $ \rho $ in the sense of the usual norm that is desired in Subsection \ref{subsec:maintheo}, whereas we should allow some overestimation for $ \rho $ (see, Table \ref{table:allen} for estimation results).

We used \cite[Theorem 2.3]{kimura1999on}  to obtain an explicit interpolation error constant $ C(M) $ ($ M\geq 1 $) satisfying
\begin{align}
\label{eq:interpolation}
\left\|v-P_{M}v\right\|_{V}\leq C(M)\left\|\Delta v\right\|_{L^{2}} \text{~~~for~all~} v \in V \cap H^2(\Omega),
\end{align}
where the orthogonal projection $ P_M $ from $ V $ to $ V_M $ is defined as
\begin{align*}
(v- P_M v,v_M)_V = 0 \text{~~~for all~} v \in V \text{~and~} v_M \in V_M.
\end{align*}
The interpolation error constant $ C^{\tau}(M) $ ($ M\geq 1 $) corresponding to the generalized norm \eqref{tau_norm} is defined as
\begin{align}
	\label{eq:interpolation_tau}
	\left\|v-P^{\tau}_{M}v\right\|_{\tau}\leq C^{\tau}(M)\left\|-\Delta v + \tau v\right\|_{L^{2}} \text{~~~for~all~} v \in V \cap H^2(\Omega),
\end{align}
where $ P^{\tau}_M $ is the orthogonal projection from $ V $ to $ V_M $ corresponding to \eqref{tau_norm} that satisfies
\begin{align*}
	(v- P^{\tau}_M v,v_M)_{\tau} = 0 \text{~~~for all~} v \in V \text{~and~} v_M \in V_M.
\end{align*}
This generalized constant $ C^{\tau}(M) $ can be estimated from $ C(M) $ via
\begin{align}
	\label{eq:CMtau}
	C^{\tau}(M) \leq C(M) \sqrt{1+\tau C(M)^{2}};
\end{align}
see \cite[Remark A.4]{tanaka2017numerical}.

This constant $ C^{\tau}(M) $ was used to obtain $ K $, a key constant for error estimation introduced below.
The lower bounds for $ \lambda(\Omega_0) $ were estimated using Corollary \ref{coro:lower_eigen}.

We proved the existence of solutions $u$ of the D-problem \eqref{main:fpro} (that is, weak solutions of \eqref{eq:allen}) in $ \overline{B}(\hat{u},\rho,\|\cdot\|_V) $ and $ \overline{B}(\hat{u},\sigma,\|\cdot\|_{L^{\infty}}) $ given approximate solutions $ \hat{u} $ constructed as \eqref{eq:uhat}.
The proof was achieved by combining the methods described in \cite{plum2008} and \cite{tanaka2014verified}.
On the basis of \cite[Theorem 1]{plum2008} , we obtained $H_{0}^{1}$-error estimates $ \rho $.
The required constants $\delta$ and $K$ and function $g$ in the theorem were computed as follows:
\begin{itemize}
	\setlength{\parskip}{1pt}
	\setlength{\itemsep}{1pt}
	\item $\delta$ was evaluated as $\delta \leq C_{2}\|\Delta\hat{u}+\varepsilon^{-2}(\hat{u}-\hat{u}^3)\|_{L^{2}}$ with $ C_2 = (2\pi^2+\tau)^{-\f{1}{2}} $.
	This $ L^2$-norm was computed using a numerical integration method with strict estimation of rounding errors \cite{kashiwagikv}.
	\item $K$, the norm of the inverse operator, was computed using the method described in \cite{tanaka2014verified}, with $ C^{\tau}(M_K) $ defined above given $ M_K \geq 1 $.
	\item $ g $ was taken as $g(t)=6 \varepsilon^{-2} C_{4}^{3} t\left(\|\hat{u}\|_{L^{4}(\Omega)}+C_{4} t\right)$; see \cite[Subsection 4.4]{plum2008}  for the construction of $ g $.
	An upper bound for $C_4$ was evaluated using the smaller estimation from \cite[Corollary A.2]{tanaka2017sharp}  and \cite[Lemma 2]{plum2008}   (see Corollary \ref{coro:roughembedding} and Theorem \ref{theo:plum_embedding}).
	Although \cite[Corollary A.2]{tanaka2017sharp}  estimates $C_4$ in the sense of the usual norm $ \|\cdot\|_V $,
	it becomes an upper bound for the embedding constant with the generalized norm \eqref{tau_norm} because $ \|\cdot\|_V \leq  \|\cdot\|_{\tau} $ for any nonnegative $ \tau $.
\end{itemize}
The solution $u \in \overline{B}(\hat{u},\rho,\|\cdot\|_V) $ of \eqref{eq:allen} has $H^{2}$-regularity
because problem \eqref{eq:poisson} subject to the zero-Dirichlet boundary condition
has a unique solution $u\in H^{2}(\Omega)$ for each $h \in L^2(\Omega)$, such as when $\Omega$ is a bounded convex polygonal domain $($again, see {\rm \cite{grisvard2011elliptic}}$) $.
Therefore, to obtain an $L^{\infty}$-error $ \sigma $, we used the following bound for the embedding $ H^2(\Omega) \hookrightarrow L^{\infty}(\Omega) $ provided in \cite[Theorem 1, Corollary 1]{plum1992explicit}.

\begin{theo}[\cite{plum1992explicit}] \label{theo:linfembedding}
There exist constants $c_0$, $c_1$, $c_2$ dependent on $\Omega$ such that,
for all $u\in H^{2}\left(\Omega\right)$,
\begin{align*}
\|u\|_{L^{\infty}}\le c_{0}\|u\|_{L^2}+c_{1}\|\nabla u\|_{L^2}+c_{2}\|u_{xx}\|_{L^2},
\end{align*}
where $u_{xx}$ denotes the Hesse matrix of $u$.
\end{theo}
\begin{rem}
When $N=2$, the norm of the Hesse matrix of $u$ is precisely defined by
\begin{align*}
\|u_{xx}\|_{L^2}=\sqrt{\sum_{i,j=1}^{2}\left\|\frac{\partial^{2}u}{\partial x_{i}\partial x_{j}}\right\|^{2}_{L^2}}.
\end{align*}
Moreover, when $\Omega$ is polygonal, we have $\left\|u_{xx}\right\|_{L^2}=\left\|\Delta u\right\|_{L^2}$ for all $u\in H^{2}(\Omega)\cap V$~$($see, for example, {\rm \cite{grisvard2011elliptic}}$)$.
\end{rem}
Explicit values of $c_0$, $c_1$, $c_2$ were provided in \cite{plum1992explicit} for $N=2,3$.
The constants displayed in example set (2) on p. 42 of \cite{plum1992explicit} can be directly used for our case where $\Omega=(0,1)^2$.
	Let us write the solution $u \in \overline{B}(\hat{u},\rho,\|\cdot\|_V) $ as $u = \hat{u} + \rho w$ with some $ w\in V$, $\|w\|_V\leq 1$.
By applying Theorem \ref{theo:linfembedding} to the error $ \rho w = u-\hat{u} \in H^2(\Omega)$, we have
\begin{align*}
\left\|u-\hat{u}\right\|_{L^{\infty}}&=\rho\left\|\omega\right\|_{L^{\infty}}\\
&\leq\rho\left(c_{0}\left\|\omega\right\|_{L^2}+c_{1}\left\|\omega\right\|_{V}+c_{2}\left\|\Delta\omega\right\|_{L^2}\right)\\
&\leq\rho\left(c_{0}C_{2}+c_{1}+c_{2}\left\|\Delta\omega\right\|_{L^2}\right).
\end{align*}
The last term $\left\|\Delta\omega\right\|_{L^{2}(\Omega)}$ is estimated via
\begin{align*}
\rho\left\|\Delta\omega\right\|_{L^{2}}
= &\left\|f\left(\hat{u}+\rho\omega\right)+\Delta\hat{u}\right\|_{L^{2}}\\
\leq& \left\|f\left(\hat{u}+\rho\omega\right)-f\left(\hat{u}\right)\right\|_{L^{2}}+\left\|\Delta\hat{u}+f\left(\hat{u}\right)\right\|_{L^{2}},
\end{align*}
where we write $f(t)=\varepsilon^{-2}(t-t^3)$.
Then, the left integral is calculated as
\begin{align}
\label{eq:linfmidterm}
\left\|f\left(\hat{u}+\rho\omega\right)-f\left(\hat{u}\right)\right\|^2_{L^{2}}
= \varepsilon^{-4} \rho^2 \left\| w (1 - 3\hat{u}^2 - 3\rho\hat{u}w - \rho^2 w^2) \right\|^2_{L^{2}}.
\end{align}
Using H\"{o}lder's inequality, we have
\begin{align*}
\left\|f\left(\hat{u}+\rho\omega\right)-f\left(\hat{u}\right)\right\|^2_{L^{2}}
\leq &\varepsilon^{-4} \rho^2 \|w\|^2_{L^3} \left\| 1 - 3\hat{u}^2 - 3\rho\hat{u}w - \rho^2 w^2 \right\|^2_{L^{6}}\\
\leq &\varepsilon^{-4} \rho^2 C_3^2 (1 + 3 \|\hat{u}\|^2_{L^{12}} + 3 \rho C_{12} \|\hat{u}\|_{L^{12}} + \rho^2 C_{12}^2)^2.
\end{align*}
Thus, we have the following $L^{\infty}$-estimation
\begin{align}
&\left\|u-\hat{u}\right\|_{L^{\infty}}\leq c_{0}C_{2}\rho+c_{1}\rho+
c_{2}\left(\rho\varepsilon^{-2}C_{3}\left(1+3\left\|\hat{u}\right\|_{L^{12}}^{2}\right.\right.\nonumber \\
&~~~~~~~~~~~~~~~~~\left.\left.
+3\rho C_{12}\left\|\hat{u}\right\|_{L^{12}}+\rho^{2}C_{12}^{2}\right)+\left\|\Delta\hat{u}+\varepsilon^{-2}(\hat{u}-\hat{u}^{3})\right\|_{L^{2}}\right). \label{eq:linffin}
\end{align}

\begin{rem}
	Inequality \eqref{eq:linffin} was used in our computations.
	However, different estimates of the right-side norm of \eqref{eq:linfmidterm} are possible.
	One such example is to calculate
	\begin{align*}
	\left\| w (1 - 3\hat{u}^2 - 3\rho\hat{u}w - \rho^2 w^2) \right\|^2_{L^{2}}
	\leq \|w\|^2_{L^4} \left\| 1 - 3\hat{u}^2 - 3\rho\hat{u}w - \rho^2 w^2 \right\|^2_{L^{4}}.
	\end{align*}
	Other than this, expanding $w^2 (1 - 3\hat{u}^2 - 3\rho\hat{u}w - \rho^2 w^2)^2$ and applying H\"{o}lder's inequality to each term need somewhat tedious calculations but would give a better estimation.
	In this case, the maximal exponent $p$ required for the embedding constant $C_p$ is reduced to 6.
\end{rem}

Table \ref{table:allen} shows the verification results for \eqref{eq:allen}.
The values in rows $\tau$, $C(M_K)$, $\delta$, $K$, $\rho$, $\sigma$, $ |\Omega_{0}| $, and $ \varepsilon^{-2} $ represent strict upper bounds in decimal form; for instance, {\rm 6.0e-03} means $ 6.0 \times 10^{-3} $. The values in row $ \lambda_{1}(\Omega_0) $ are lower bounds, which were estimated using Corollary \ref{coro:lower_eigen}.
Integers $ M_u $, $ M_K $, and $ 2^m $ are displayed as strict integers.
Volumes $ |\Omega_0| $ were estimated by dividing $ \Omega $ into $ 2^m $ smaller congruent squares and implementing interval arithmetic on them to confirm $ (\hat{u}+\sigma)(\hat{u}-\sigma) \leq 0 $.
Approximate solutions $ \hat{u} $ and the corresponding defect bounds $ \delta $ were computed in double-double precision using the data type ``dd'' or ``interval$ < $ dd $ > $'' provided in the kv Library {\rm \cite{kashiwagikv}}.
Although the values in row $ \rho  $ represent the error bounds in the sense of the norm \eqref{tau_norm} for corresponding $ \tau $'s,
these can be regarded as upper bounds for them in the sense of the usual norm $ \|\cdot\|_V $ required in Subsection \ref{subsec:maintheo}.  

In all cases, Corollary \ref{coro:main} estimated the numbers of nodal domains under the condition \eqref{cond:coro1}. This indicated that $\#\text{\rm N.N.D.}(u)$ for type (A) was 1 or 2.
The reason why $\#\text{\rm N.N.D.}(u)$ was not strictly determined is that it is difficult to determine whether the negative nodal domains that appear to be composed of two parts are connected or not through the boundary (see Fig.~\ref{fig:nodal_line} (A)).

For solutions of type (B), neither $\#\text{\rm P.N.D.}(u)$ nor $\#\text{\rm N.N.D.}(u)$ was strictly determined.
However, we can determine both $\#\text{\rm P.N.D.}(u)$ and $\#\text{\rm N.N.D.}(u)$ to be two by considering the symmetry of the solutions and the topology of nodal lines (that is, ``how the lines intersect'') in the following discussion:
Let us define $v(x_1,x_2):=-u(x_2,x_1)$ so that $v$ is also a solution of \eqref{eq:allen}.
We define $\hat{v}(x_1,x_2):=-\hat{u}(x_2,x_1)$ for each approximate solution $\hat{u}$ for type (B), assuming that $\|\hat{v}+\hat{u}\|_V \leq \eta$ for small $\eta>0$.
Actually, we confirmed this inequality when selecting $\eta=$1e-15 in all cases for type (B).
Then, we have $\|v-\hat{u}\|_V \leq \rho + \eta$.
We again checked the conditions required by \cite[Theorem 1]{plum2008} with $\rho$ replaced by $\rho + \eta$, thereby proving the uniqueness of the solution $u$ in $\overline{B}(\hat{u},\rho + \eta,\|\cdot\|_V)$.
Therefore, we concluded $u=v$ and thus $u(0.5,0.5)=-u(0.5,0.5)=0$.
Similarly, the symmetry of the solutions was confirmed with respect to the lines $x=0.5$ and $y=0.5$ by considering the transformed functions $u(1-x,y)$ and $u(x,1-y)$, respectively.
Hence, by considering the topology of nodal lines, we have confirmed $\#\text{\rm P.N.D.}(u)=\#\text{\rm N.N.D.}(u)=2$.

For solutions of type (C), Corollary \ref{coro:main} strictly determined both $\#\text{\rm P.N.D.}(u)$ and $\#\text{\rm N.N.D.}(u)$ without a topological consideration such as that for type (B).
These solutions are special because the inner nodal line does not touch the original boundary $ \partial \Omega $ (see Fig.~\ref{fig:nodal_line} (C)).
In this sense, we can regard the inner nodal line can as a ``new'' nontrivial Dirichlet boundary.
To our knowledge, the existence of such solutions of problem \eqref{eq:allen} has not been proved.
Our method confirmed this existence using the methods in \cite{plum2008,plum1992explicit,tanaka2014verified} and Corollary \ref{coro:main}.

From the above verification results, we can determine the location of the nodal line of $ u $.
Fig.~\ref{fig:nodal_line} shows verified nodal lines of the solutions {\rm (A)}, {\rm (B)}, and {\rm (C)} for $ \varepsilon = 0.08 $.
We confirmed that $ (\hat{u}+\sigma)(\hat{u}-\sigma) \leq 0 $ on the red squares displayed in Fig.~\ref{fig:nodal_line}.
For ease of viewing, these are drawn with rough accuracy by dividing the domain $ \Omega $ into $ 2^{12} $ smaller congruent squares and implementing interval arithmetic on each.
In Fig.~\ref{fig:accurate_nodal_line}, we display a more accurate nodal line via division into $ 2^{16} $ smaller congruent squares.
Our method proved the nonexistence of nodal domains of $u$ in the union of the red squares for each solution.
Simultaneously, the sign of $ u $ is strictly determined in the blanks.

\begin{table}[t]
	\caption{Verification results for \eqref{eq:allen} on $\Omega=(0,1)^{2}$.}
	\label{table:allen}
	\begin{center}
		\fontsize{6pt}{10pt}\selectfont
		
		\begin{tabular}{l|lll|lll|lll}
			\hline
			ID&
			\multicolumn{3}{l|}{(A)}&
			\multicolumn{3}{l|}{(B)}&
			\multicolumn{3}{l}{(C)}\\

			$\varepsilon$&
			0.1&
			0.08&
			0.06&
			0.1&
			0.08&
			0.06&
			0.1&
			0.08&
			0.06\\			
			
			\hline	\hline
			
			$M_u$&
			100&
			100&
			100&
			80&
			80&
			80&
			100&
			100&
			100\\

			$M_K$&
			80&
			80&
			80&
			80&
			80&
			80&
			50&
			100&
			100\\
			$\tau$&
			0&
			102.3&
			436.3&
			0&
			126.1&
			481.7&
			0&
			217.5&
			545.6\\
			
			$C^{\tau}(M_K)$&
			6.0e-03&
			6.0e-03&
			6.1e-03&
			6.0e-03&
			6.1e-03&
			6.1e-03&
			9.4e-03&
			4.9e-03&
			4.9e-03\\
			$\delta$&
			1.6e-16&
			5.4e-13&
			2.8e-08&
			1.5e-16&
			1.1e-12&
			1.5e-08&
			1.5e-16&
			3.7e-15&
			7.2e-13\\

			$K$&
			1113&
			10.4&
			53.4&
			263&
			12.9&
			13.4&
			261&
			14.8&
			12.3\\		
			
			$ \rho$&
			4.0e-14&
			5.1e-13&
			6.9e-08&
			
			8.8e-15&
			1.2e-12&
			8.6e-09&
			
			8.8e-15&
			4.1e-15&
			3.8e-13\\

			$ \sigma$&
			1.5e-13&
			1.1e-11&
			3.2e-06&
			
			3.2e-14&
			1.9e-11&
			3.3e-07&
			
			3.2e-14&
			5.5e-14&
			1.4e-11\\
			
			$2^m$&
			$2^{20} $&
			$2^{20} $&
			$2^{20} $&
			
			$2^{20} $&
			$2^{22} $&
			$2^{24} $&
			
			$2^{20} $&
			$2^{20} $&
			$2^{20} $\\		
			
			$|\Omega_0|$&
			9.5e-02&
			1.1e-02&
			1.1e-02&
			
			4.6e-02&
			2.9e-02&
			1.6e-02&
			
			9.0e-03&
			1.1e-02&
			1.4e-02\\
			$\lambda_{1} (\Omega_0) \geq$&
			664.6&
			625.7&
			597.6&
			
			137.5&
			222.8&
			396.3&
			
			704.7&
			574.1&
			459.0\\
			
			$\varepsilon^{-2}$&
			100.0&
			156.3&
			277.8&
			
			100.0&
			156.3&
			277.8&
			
			100.0&
			156.3&
			277.8\\

			$\text{\rm \#P.N.D.}(u)$&
			\multicolumn{3}{l|}{1}&
			\multicolumn{3}{l|}{1--2}&
			\multicolumn{3}{l}{1}\\

			$\text{\rm \#N.N.D.}(u)$&
			\multicolumn{3}{l|}{1--2}&
			\multicolumn{3}{l|}{1--2}&
			\multicolumn{3}{l}{1}\\

			$\text{\rm \#N.D.}(u)$&
			\multicolumn{3}{l|}{2--3}&
			\multicolumn{3}{l|}{2--4}&
			\multicolumn{3}{l}{2}\\

			\hline
		\end{tabular}
	\end{center}
	~\\
	\scriptsize{
	$\varepsilon$: positive parameter in \eqref{eq:allen}.\\
	$M_u$: number of basis functions for constructing approximate solution $ \hat{u} $; see \eqref{eq:uhat}.\\
	$M_K$: number of basis functions for calculating $ K $.\\
	$\tau$: nonnegative number satisfying \eqref{select:tau}.\\
	$C^{\tau}(M_K)$: interpolation constant calculated via \eqref{eq:CMtau}.\\
	$\delta$: defect bound required in \cite[Theorem 1]{plum2008} .\\
	$K$: norm of the inverse operator required in \cite[Theorem 1]{plum2008} .\\
	$ \rho $: $ H^1_0 $-error bound.\\
	$ \sigma $: $ L^{\infty} $-error bound.\\
 	$|\Omega_0|$: volume of $\Omega_0$; $\Omega_0$ is defined just before Subsection \ref{subsec:maintheo}.\\
 	$\lambda_{1} (\Omega_0)$: first eigenvalue of $ -\Delta $ on $ \Omega_0 $ defined by \eqref{eq:defi-diri-eigenvalue}.\\
 	$\text{\rm \#P.N.D.}(u)$ ($\text{\rm \#N.N.D.}(u)$): number of positive (negative) nodal domains of $ u $; see Definition \ref{defi:nodaldomains}.\\
 	$\text{\rm \#N.D.}(u)$: number of nodal domains that satisfy $\text{\rm \#N.D.}(u)=\text{\rm \#P.N.D.}(u)+\text{\rm \#N.N.D.}(u)$.
	}
\end{table}

\newcommand{\sizee}{0.29\hsize}
\begin{figure}[t]
	\footnotesize{$\varepsilon=0.1$~}
	\begin{minipage}{\sizee}
		\begin{center}
			(A)\\
			\includegraphics[height=33 mm]{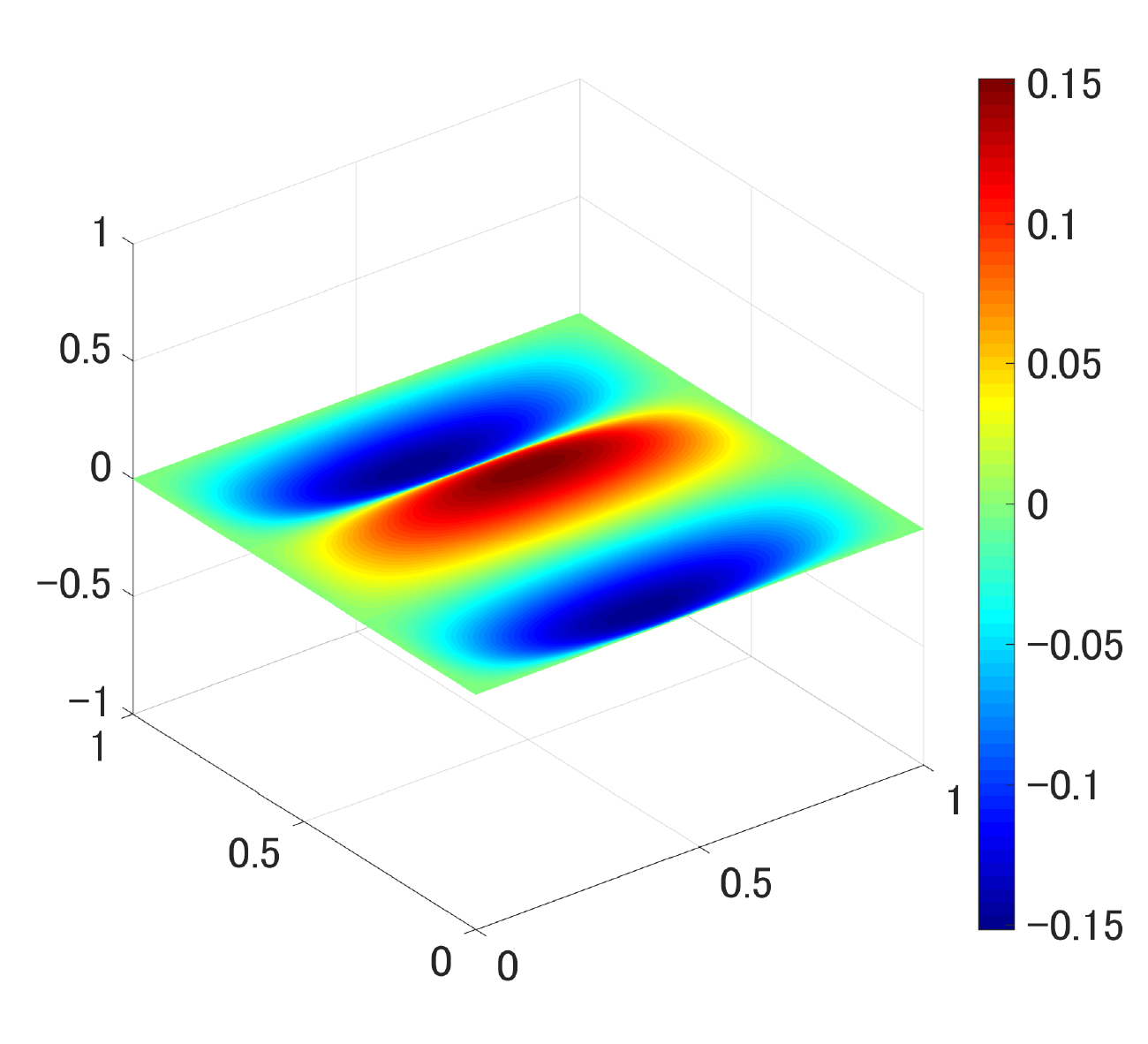}
		\end{center}
		~
	\end{minipage}
	\begin{minipage}{\sizee}
		\begin{center}
			(B)\\
			\includegraphics[height=33 mm]{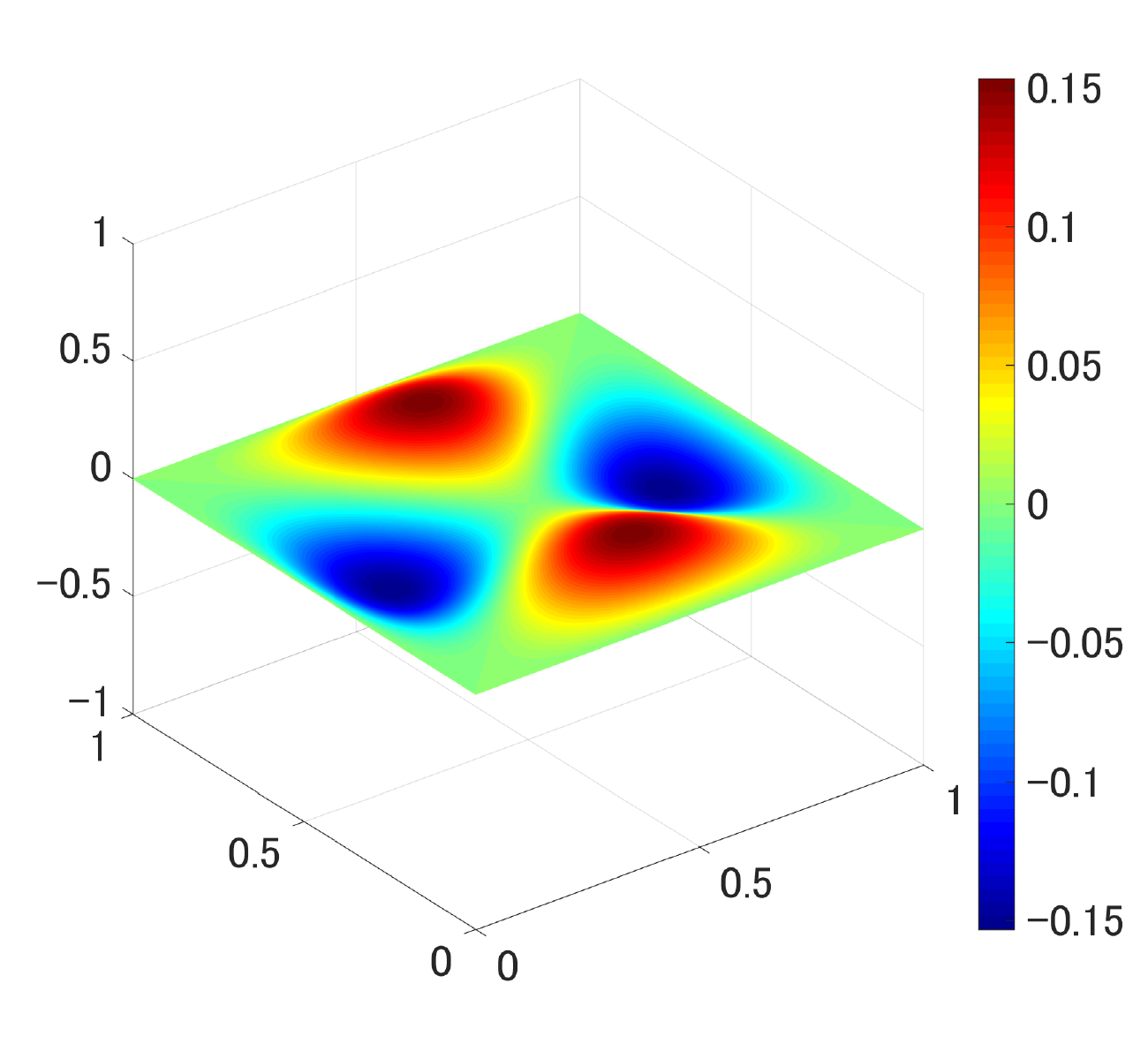}
		\end{center}
		~
	\end{minipage}
	\begin{minipage}{\sizee}
		\begin{center}
			(C)\\
			\includegraphics[height=33 mm]{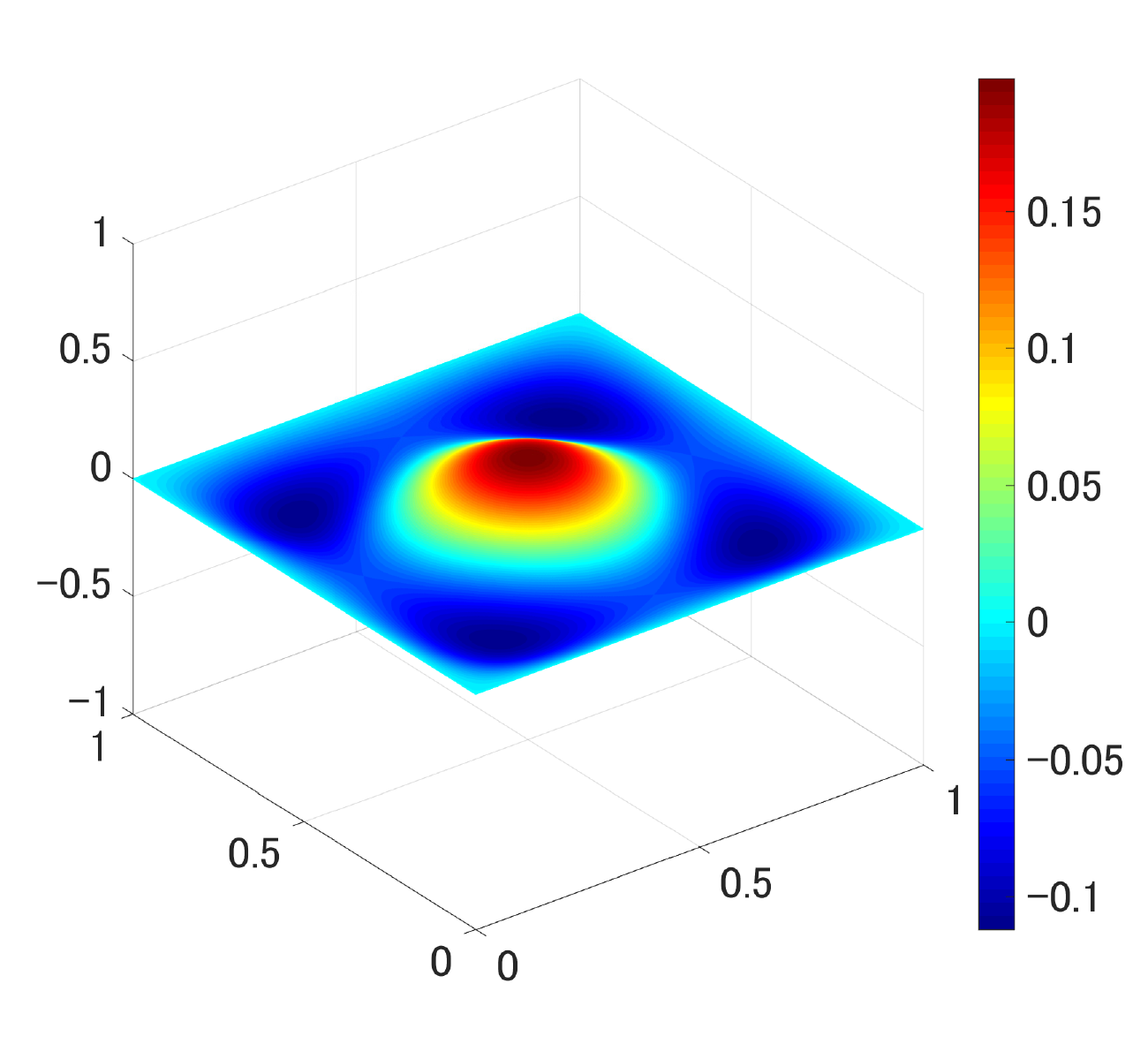}
		\end{center}
		~
	\end{minipage}
	
	\footnotesize{$\varepsilon=0.08$}
	\begin{minipage}{\sizee}
		\begin{center}
			\includegraphics[height=33 mm]{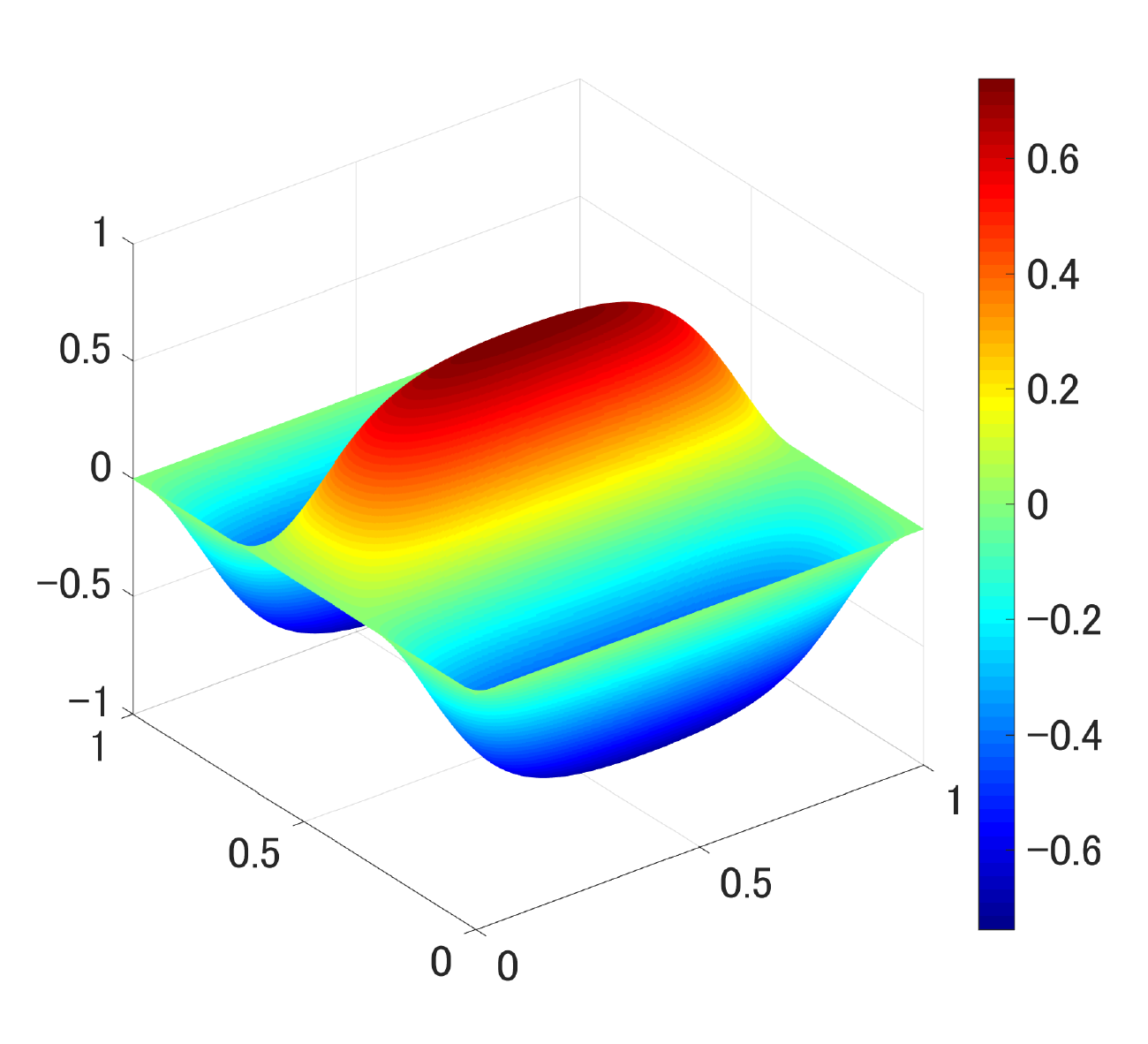}
		\end{center}
		~
	\end{minipage}
	\begin{minipage}{\sizee}
		\begin{center}
			\includegraphics[height=33 mm]{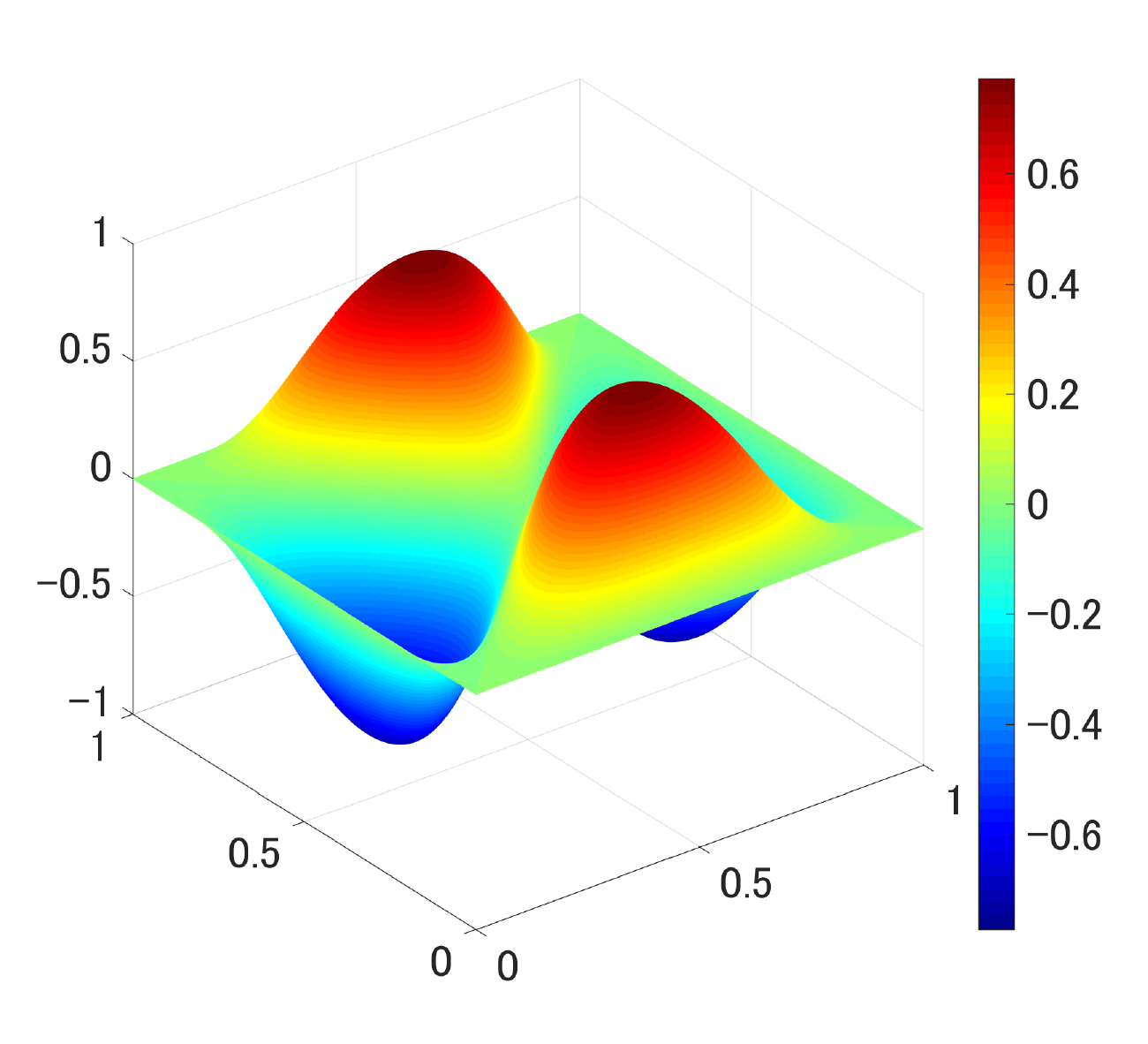}
		\end{center}
		~
	\end{minipage}
	\begin{minipage}{\sizee}
		\begin{center}
			\includegraphics[height=33 mm]{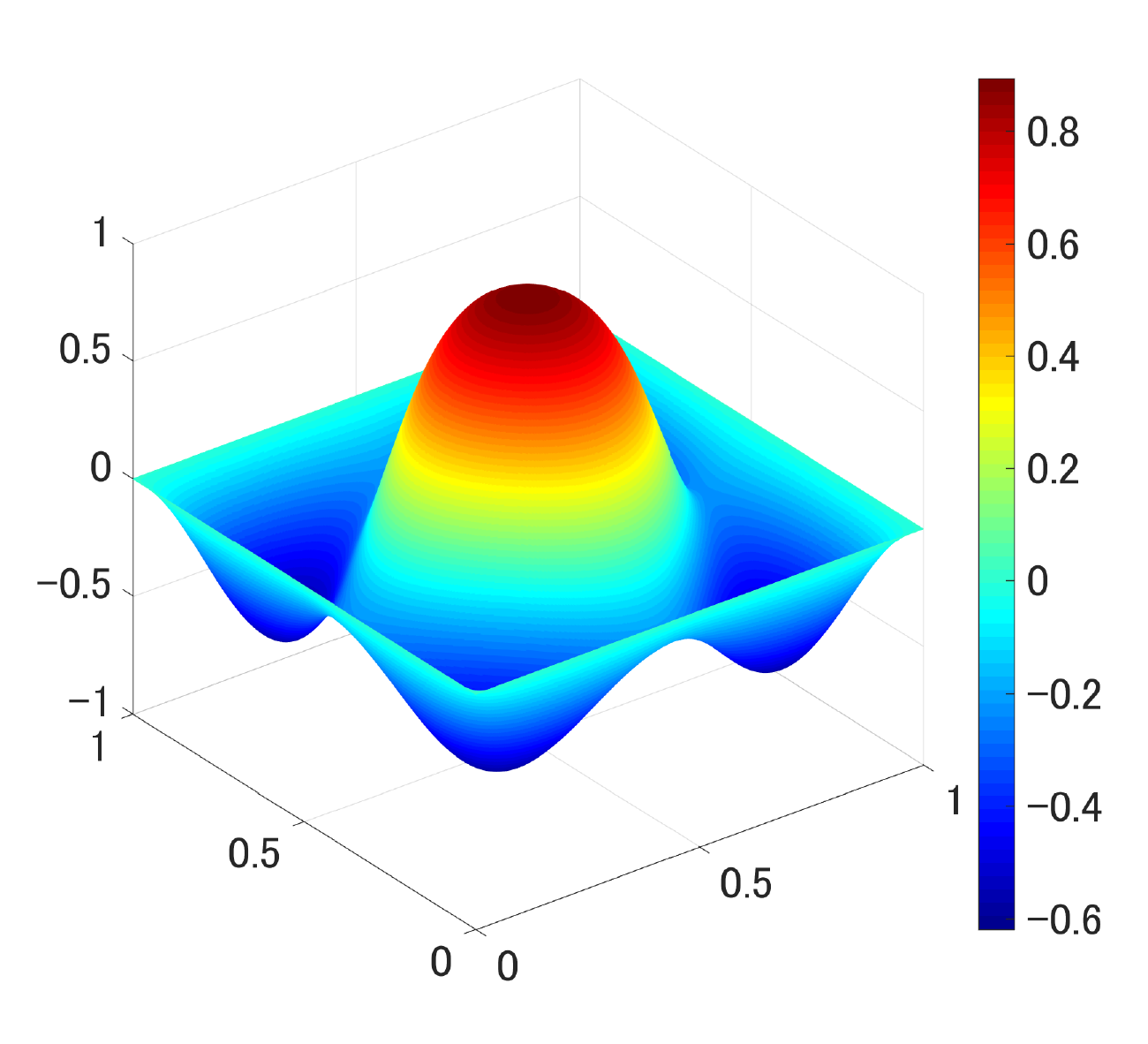}
		\end{center}
		~
	\end{minipage}
	
	\footnotesize{$\varepsilon=0.06$}
	\begin{minipage}{\sizee}
		\begin{center}
			\includegraphics[height=33 mm]{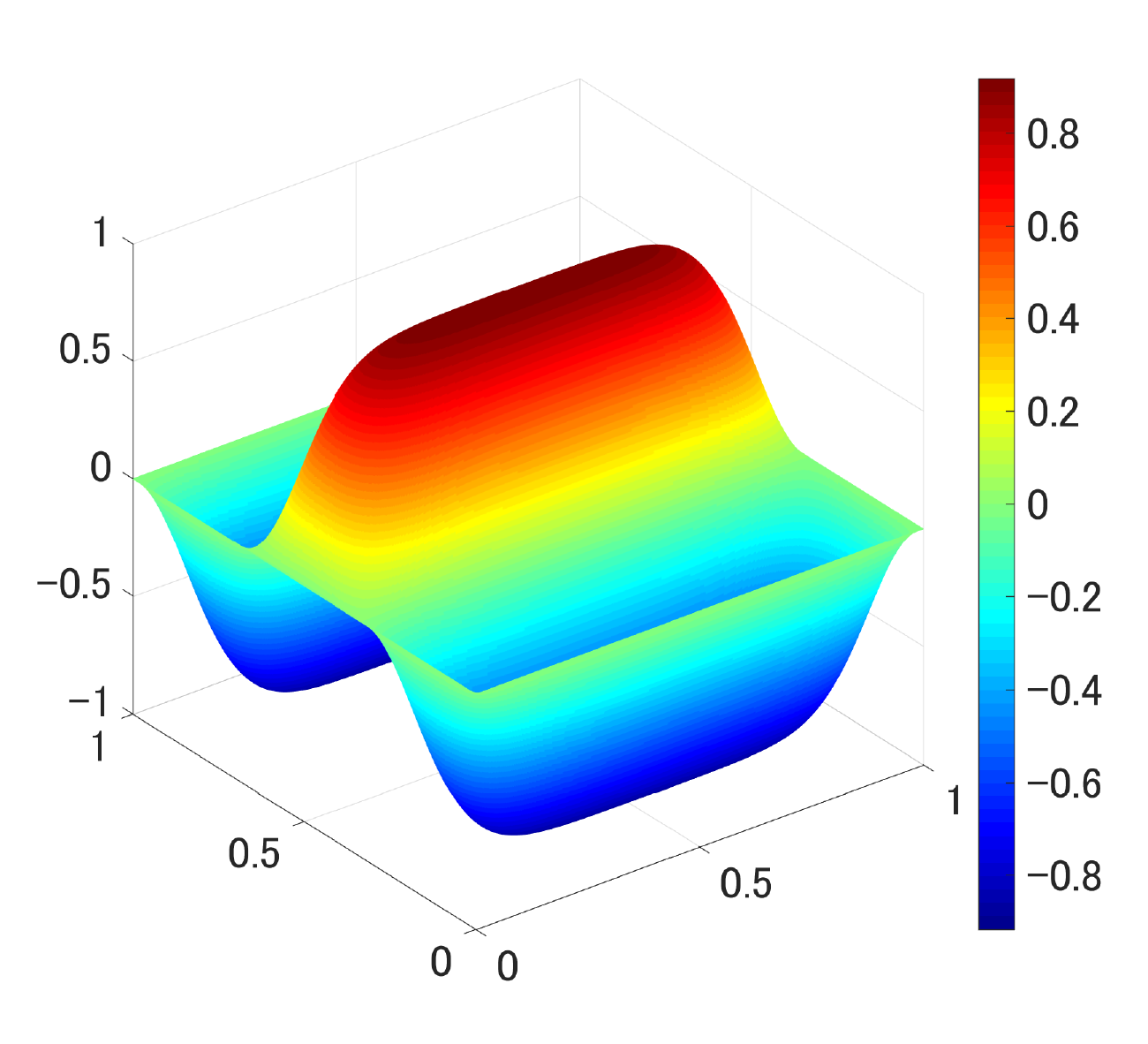}
		\end{center}
	\end{minipage}
	\begin{minipage}{\sizee}
		\begin{center}
			\includegraphics[height=33 mm]{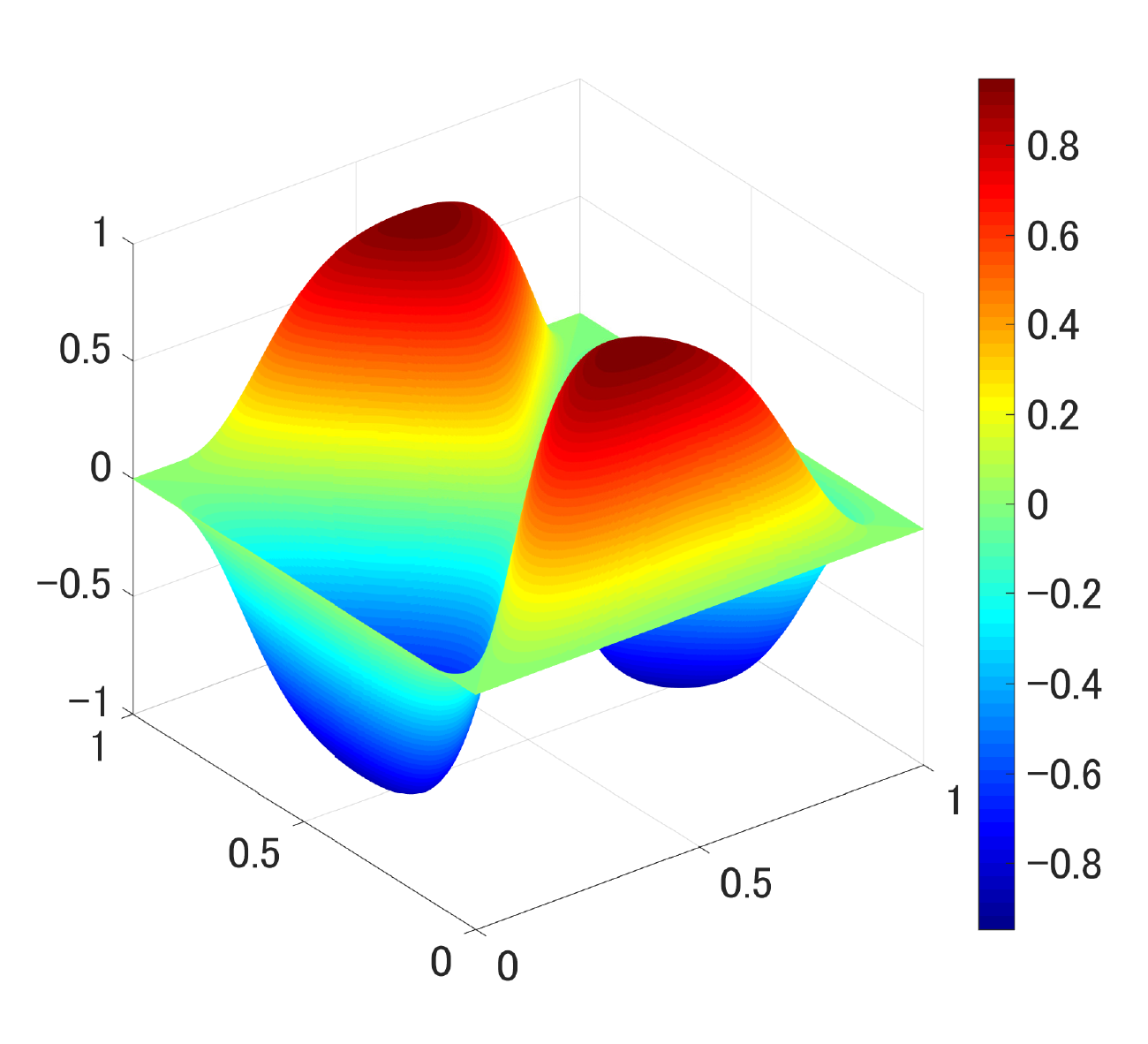}
		\end{center}
	\end{minipage}
	\begin{minipage}{\sizee}
		\begin{center}
			\includegraphics[height=33 mm]{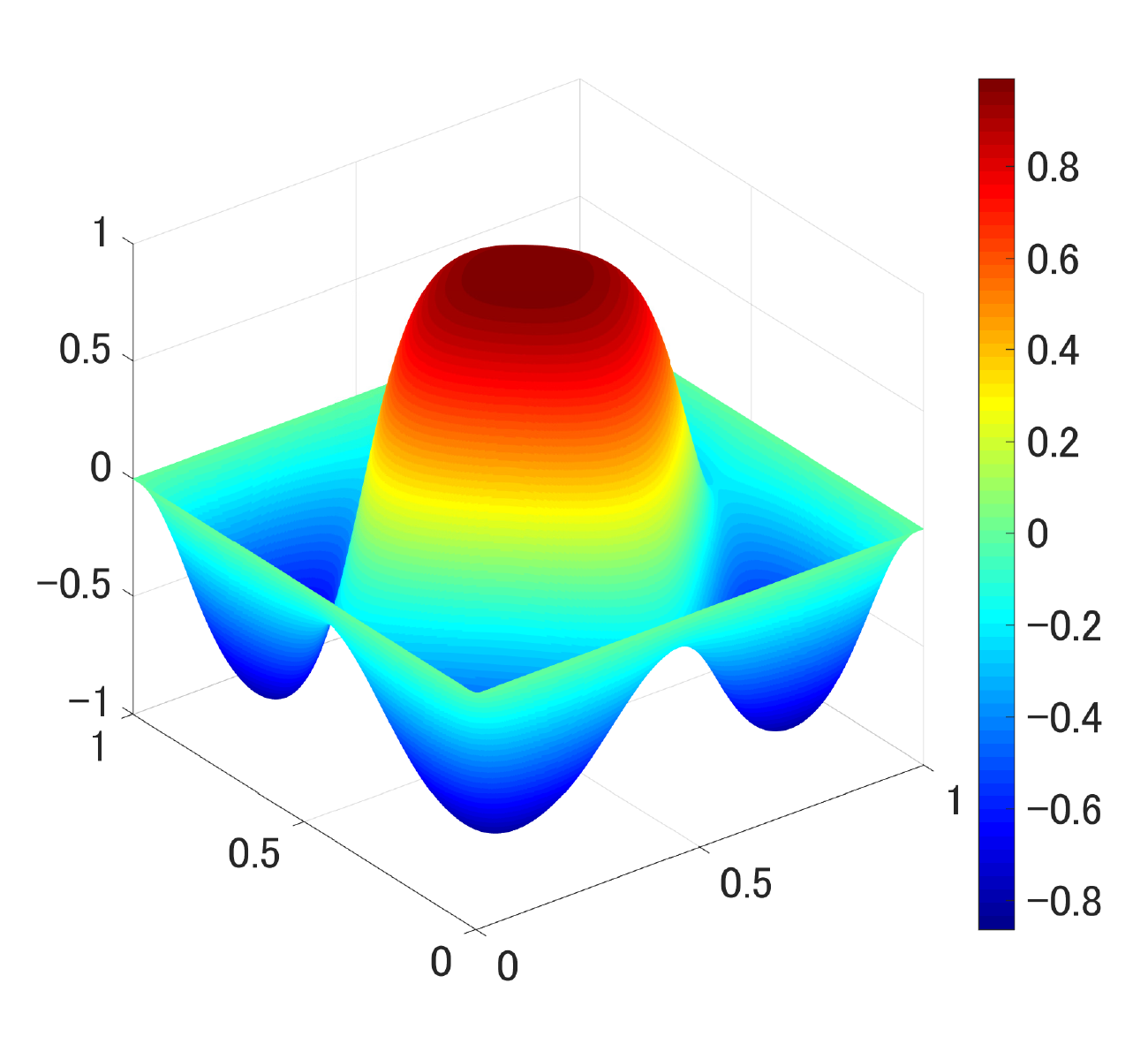}
		\end{center}
	\end{minipage}
	
	\caption{Sign-changing solutions of \eqref{eq:allen} on $\Omega=(0,1)^{2}$.}
	\label{fig:allen}
\end{figure}

\renewcommand{\sizee}{0.325\hsize}
\begin{figure}
	\begin{minipage}{\sizee}
		\begin{center}
			~~(A)\\
			\includegraphics[height=30 mm]{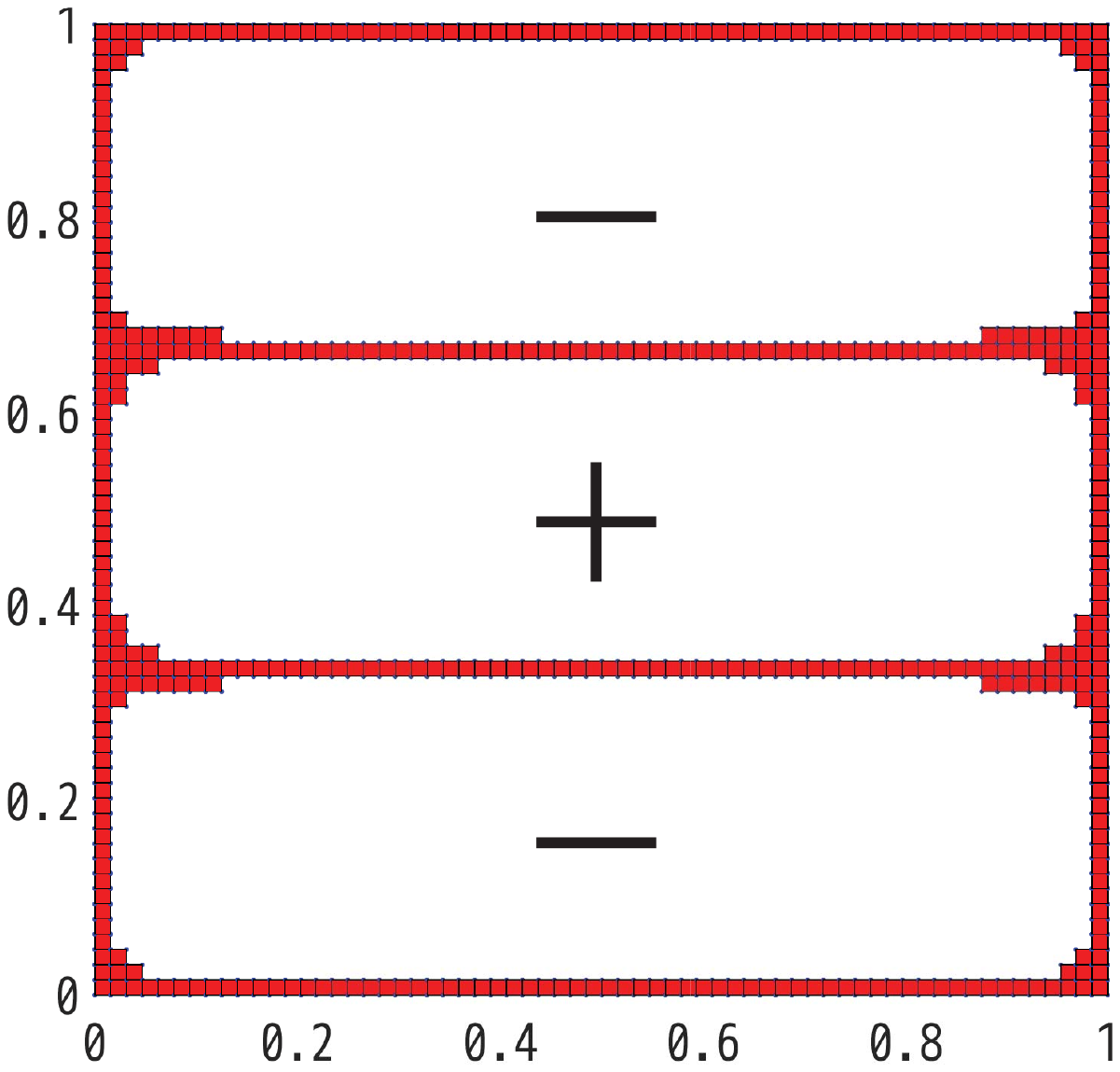}
		\end{center}
	\end{minipage}
	\begin{minipage}{\sizee}
		\begin{center}
			~~(B)\\
			\includegraphics[height=30 mm]{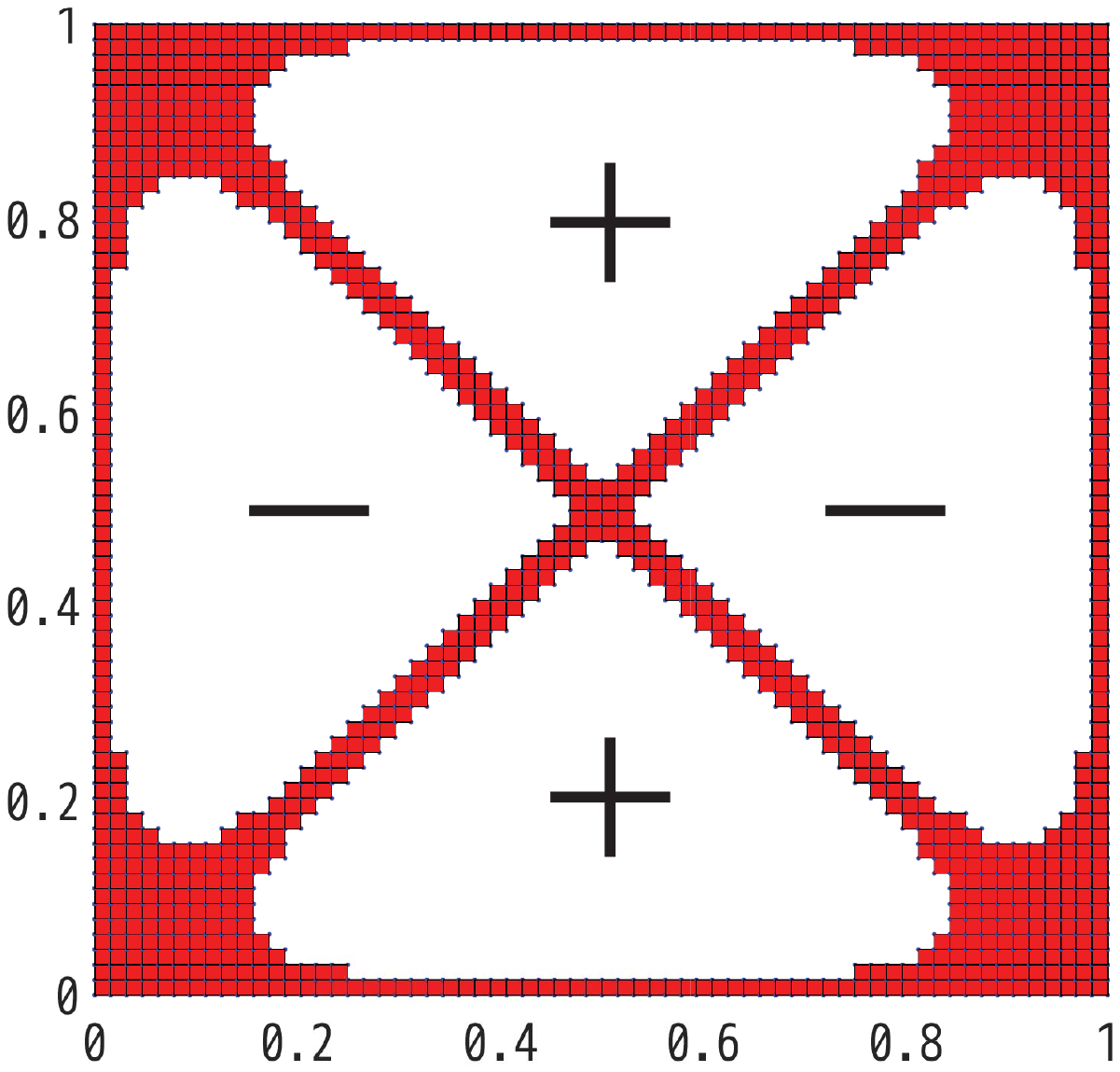}
		\end{center}
	\end{minipage}
	\begin{minipage}{\sizee}
		\begin{center}
			~~(C)\\
			\includegraphics[height=30 mm]{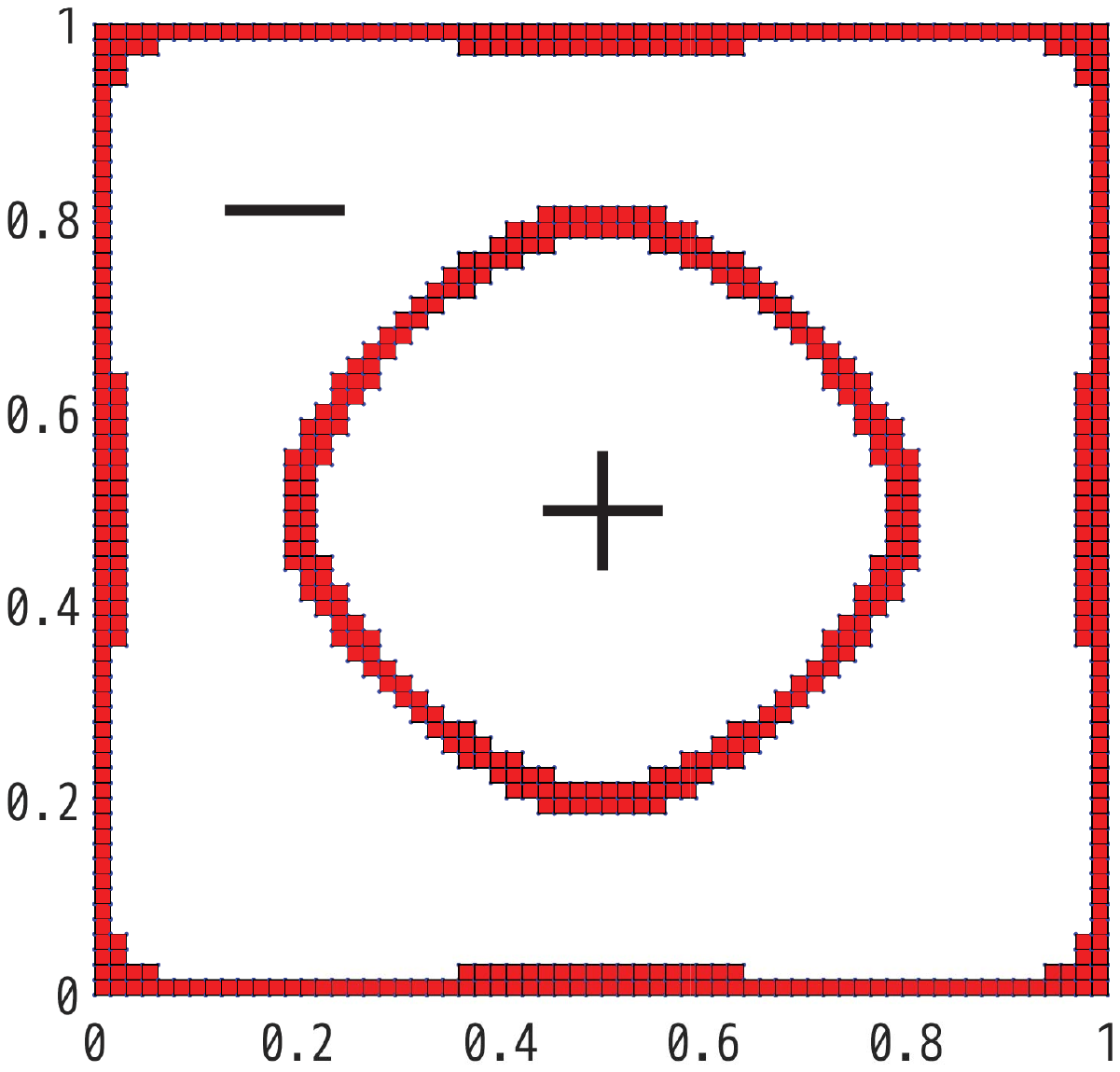}
		\end{center}
	\end{minipage}
	\caption{Verified nodal lines of the solutions {\rm (A)}, {\rm (B)}, and {\rm (C)} for $ \varepsilon = 0.08 $.
		These were drawn with rough accuracy by dividing the domain $ \Omega $ into $ 2^{12} $ smaller congruent squares and implementing interval arithmetic on each.
	}
	\label{fig:nodal_line}
\end{figure}
\begin{figure}
	\begin{minipage}{\sizee}
		\begin{center}
			\includegraphics[height=30 mm]{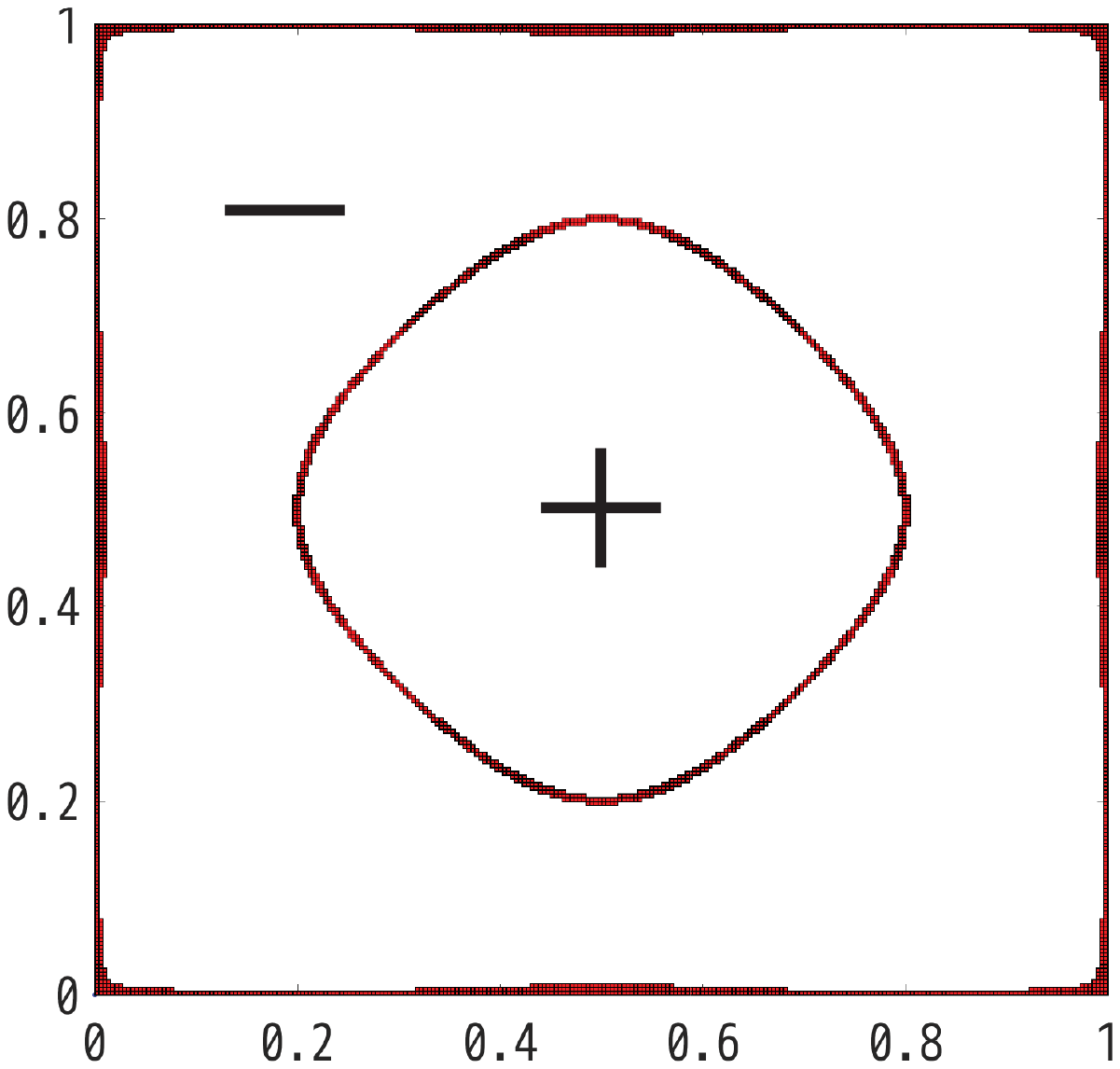}
		\end{center}
	\end{minipage}
	\begin{minipage}{\sizee}
		\begin{center}
			\includegraphics[height=30 mm]{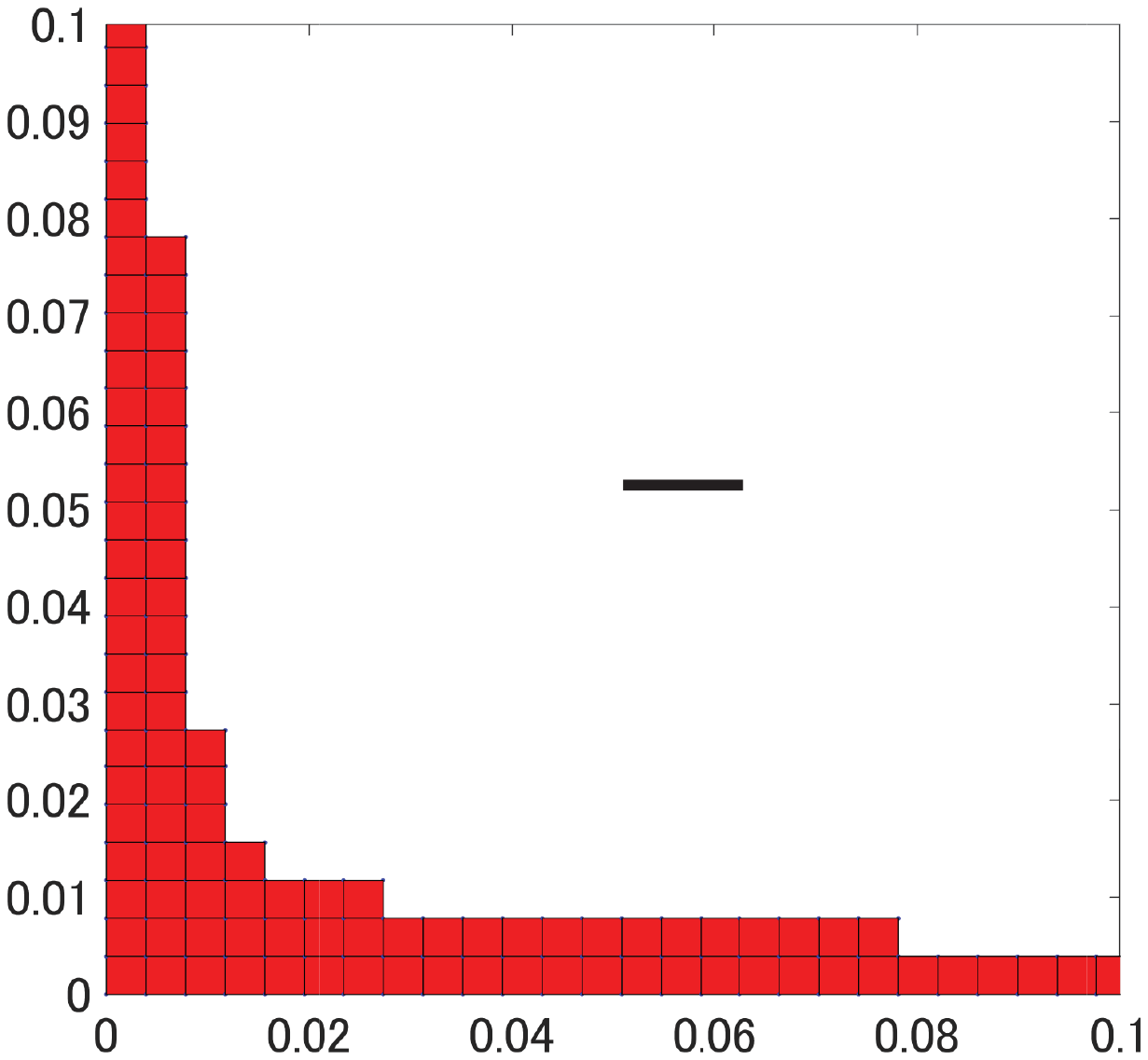}
		\end{center}
	\end{minipage}
	\begin{minipage}{\sizee}
		\begin{center}
			\includegraphics[height=30 mm]{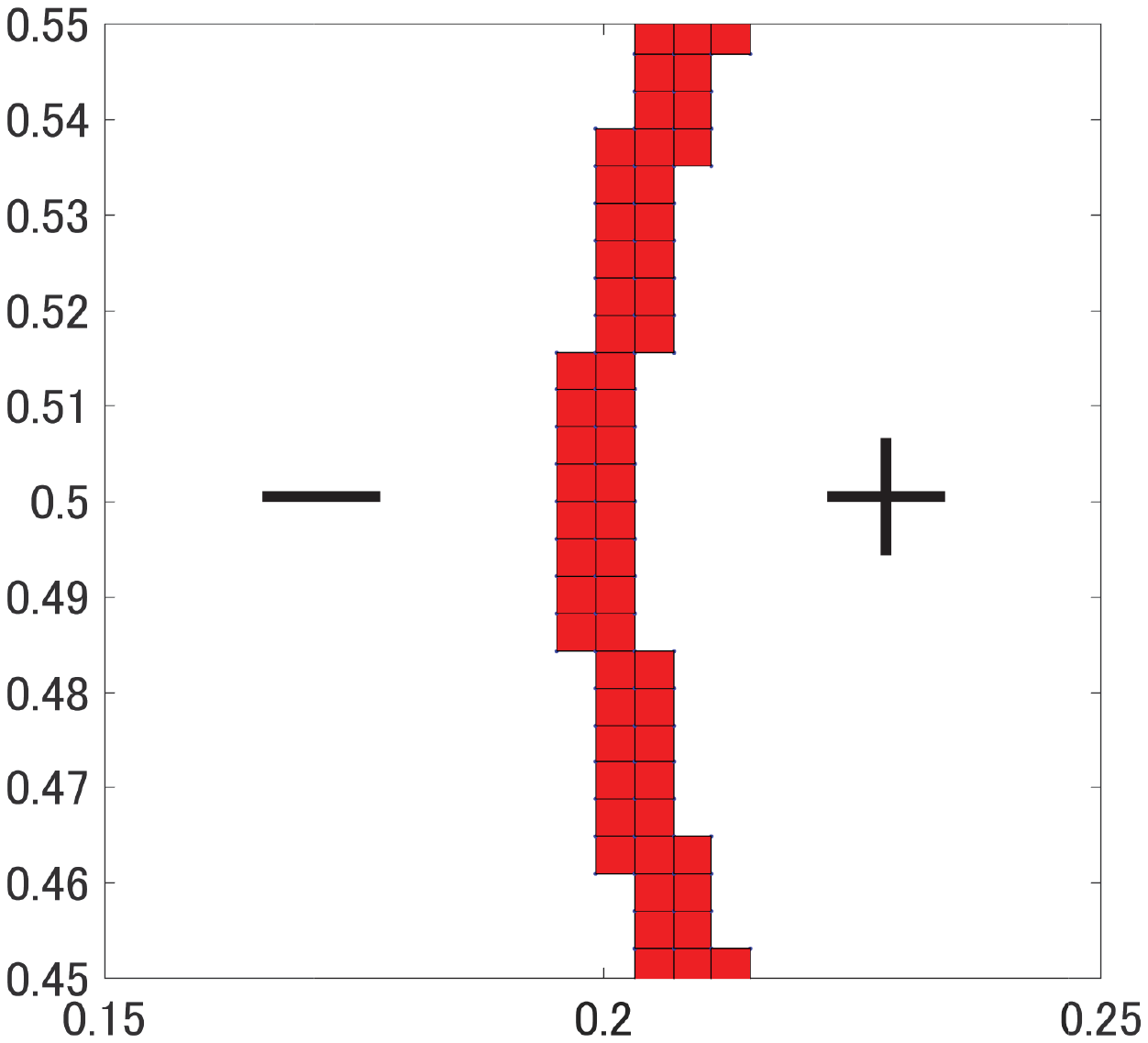}
		\end{center}
	\end{minipage}
	\caption{Accurate inclusion of the nodal line of solution {\rm (C)} with $ \varepsilon=0.08 $ $ ( $left$ ) $, and its magnifications $ ( $center and right$ ) $.
		These were drawn by dividing the domain $ \Omega $ into $ 2^{16} $ smaller congruent squares and implementing interval arithmetic on each.
	}
	\label{fig:accurate_nodal_line}
\end{figure}
\section{Extension to other boundary value conditions} \label{sec:othreboundary}
In this section, we extend the results from Section \ref{sec:diri} to Neumann \eqref{eq:mainn} and mixed \eqref{eq:mainm} boundary conditions.
Because \eqref{eq:mainm} coincides with \eqref{eq:mainn} when $ \Gamma_D = \emptyset $ and $ \Gamma_N = \partial \Omega $, we discuss the application to \eqref{eq:mainm}.
The Dirichlet problem \eqref{eq:maind} is regarded as \eqref{eq:mainn} for the special case $ \Gamma_N = \emptyset $ and $ \Gamma_D = \partial \Omega $.
Therefore, the generalization to \eqref{eq:mainm} is considered as an extension of the method provided in Section \ref{sec:diri}. 

We introduce (or replace) some required notation.
We extend the solution space $ V $ to $ V ~(= V(\Omega,\Gamma_D)) :=  \{ u \in H^1(\Omega) : u = 0~\mbox{on} ~ \Gamma_D\}$ adapting to the corresponding boundary value condition.
The inner product endowed with $ V $ should be changed according to the boundary value conditions.
When $ \Gamma_D = \emptyset $ (Neumann condition), we endow $ V $ with the inner product $(u, v)_{V}=(\nabla u, \nabla v)_{L^2}+(u, v)_{L^2}$;
otherwise (Dirichlet or mixed condition), we endow it with $(u, v)_{V}=(\nabla u, \nabla v)_{L^2}$.
The norm endowed with $ V $ is always $\| u \|_{V} = \sqrt{(u, u)_{\smash{V}}}$ regardless of the boundary conditions.
Additionally, the topological dual of $ V $ is denoted by $ V^* $.
In this function space, the weak form of \eqref{eq:mainm} is characterized by the form \eqref{main:fpro} with the same assumptions for nonlinearity $ f $ introduced in Section \ref{sec:diri}.
To avoid confusion, we call \eqref{main:fpro} corresponding to \eqref{eq:mainn} (assuming $\Gamma_D = \emptyset$ and $\Gamma_N = \partial \Omega$) the {\it N-problem}, and call \eqref{main:fpro} corresponding to \eqref{eq:mainm} (assuming $\Gamma_D \neq \emptyset$ and $\Gamma_N \neq \emptyset$) the {\it M-problem}.

We extend the definition of embedding constants.
A norm bound for the embedding $V(\Omega,\Gamma_D) \hookrightarrow L^{p+1}(\Omega)$ is denoted by $C_{p+1}~(=C_{p+1}(\Omega,\Gamma_D))$, which satisfies
\begin{align}
\label{embedding-mix}
\left\|u\right\|_{L^{p+1}(\Omega)}\leq C_{p+1}\left\|u\right\|_{V(\Omega,\Gamma_D)}~~~{\rm for~all}~u\in V,
\end{align}
where $p\in [1,\infty)$ when $N=2$ and $p\in [1,p^*]$ when $ N\geq3 $.
In the following definition \eqref{eq:eig-mix}, we assume $ \Gamma_D \neq \emptyset $;
considering this case is sufficient for completing the later discussion.
The first eigenvalue of $ -\Delta $ on $ V(\Omega,\Gamma_D) $ is denoted by $ \lambda_{1}(\Omega,\Gamma_D) $, the definition of which is
\begin{align}
\label{eq:eig-mix}
\lambda_{1}(\Omega,\Gamma_D) := \inf_{v\in V\backslash{\{0\}}} \frac{\|v\|_{V(\Omega,\Gamma_D)}^2}{\|v\|_{L^2(\Omega)}^2}.
\end{align}
\begin{lem}\label{lem:mix}
	The same argument in Lemma $ \ref{lem:positive} $ is true for the M-problem \eqref{main:fpro} with a nonempty $ \Gamma_D $, where the old notation of eigenvalue $ \lambda_1(\Omega) $ and embedding constants $ C_{p_{i}+1}(\Omega) $ is replaced with the new notation  $ \lambda_1(\Omega,\Gamma_D) $ and $ C_{p_{i}+1}(\Omega,\Gamma_D) $, respectively.
\end{lem}
\proof
Inequality \eqref{ineq:prooflemma} holds for the notational replacements.
\qed\\

The connected components of $ \mathring\Omega_0 $ are denoted by $\mathring\Omega_0^j$ ($ j=1,2,\cdots $), the number of which is assumed to be finite.
Note that $\partial \mathring\Omega_0^j \backslash \Gamma_N $ is not empty because $\partial \mathring\Omega_0^j \backslash \partial \Omega \neq \emptyset$ is ensured from $ \Omega_{0} \neq \Omega $.

Moreover, we recall our assumption: some numerical verification method succeeds in proving the existence of a solution $u \in V \cap L^{\infty}(\Omega)$ of the D-, N-, or M-problem of \eqref{main:fpro} in both balls \eqref{eq:h10ball} and \eqref{eq:linfball} in this ``extended'' setting.	
\begin{theo}\label{theo:mix}
	Let $ f $ satisfy \eqref{f:lem1} for some $ \lambda < \displaystyle \min_{j} \{\lambda_1(\mathring\Omega_0^j,\partial \mathring\Omega_0^j \backslash \Gamma_N)\} $.
	Let $ C_{p_{i}+1}=C_{p_{i}+1}(\Omega, \Gamma_D) $, $C_{p_{i}+1}^j=C_{p_{i}+1}(\mathring\Omega_0^j,\partial \mathring\Omega_0^j \backslash \Gamma_N)$, and $\lambda_1^j=\lambda_1(\mathring\Omega_0^j,\partial \mathring\Omega_0^j \backslash \Gamma_N)$.
	If we have
	\begin{align}
	\label{cond:theo2}
	\displaystyle \sum_{i=1}^{n}a_i (C_{p_{i}+1}^j)^2\left( \left\|\hat{u}\right\|_{L^{p_{i}+1}(\mathring\Omega_{0}^j)}+C_{p_{i}+1}\rho\right)^{p_{i}-1}<1 - \f{\lambda}{\lambda_1^j},
	\end{align}
	for each $ j $, then a solution $u \in V \cap L^{\infty}(\Omega)$ of the D-, N-, or M-problem of \eqref{main:fpro} existing in the intersection of balls \eqref{eq:h10ball} and \eqref{eq:linfball} satisfies \eqref{ineq:cc1} and \eqref{ineq:cc2}.
\end{theo}
\proof
We prove the nonexistence of nodal domains in $ \mathring\Omega_0^j $ for every $ j $, as well as in the proof of Theorem \ref{theo:main}.
To achieve this, we consider the following two cases.

\subsubsection*{Case 1 --- when $ V(\mathring\Omega_0^j,\partial \mathring\Omega_0^j \backslash \Gamma_N)=H^1_0(\mathring\Omega_0^j) $}
In this case,
almost the same discussion as in the proof of Theorem \ref{theo:main} can be applied (see $ \Omega_0^3 $ in Fig.~\ref{fig:domain0}). 

Suppose that there exists a subdomain $ \Omega' \subset  \mathring\Omega_0^j$ such that $ u|_{\Omega'} \in H^1_0(\Omega')~(\subset H^1_0(\mathring\Omega_0^j)) $ is a solution of the D-problem \eqref{main:fpro} with the replacement $ \Omega \rightarrow \Omega' $.
We express $u\in V\,(=V(\Omega,\Gamma_D))$ as $ u=\hat{u}+ \rho\omega$, where $\omega\in V $ satisfies $\left\|\omega\right\|_{V}\leq 1$.
This ensures that, for $p\in (1,p^*)$,
\begin{align*}
\left\|u\right\|_{L^{p+1}(\Omega')}
\leq
\left\|\hat{u}\right\|_{L^{p+1}(\Omega')}+ C_{p+1}\rho	
\end{align*}
because $\left\|\omega \right\|_{L^{p+1}(\Omega')}\leq \left\|\omega \right\|_{L^{p+1}(\Omega)} \leq C_{p+1}\left\|\omega \right\|_{V}\leq C_{p+1}$, where $ C_{p+1}=C_{p+1}(\Omega,\Gamma_D) $.
It readily follows from $ \left\|\hat{u} \right\|_{L^{p+1}(\Omega')} \leq \left\|\hat{u} \right\|_{L^{p+1}(\mathring\Omega_0^j)} $ that
\begin{align}
\left\|u \right\|_{L^{p+1}(\Omega')}
\leq
\left\|\hat{u} \right\|_{L^{p+1}(\mathring\Omega_0^j)}+ C_{p+1}\rho.
\label{u:eval2}
\end{align}	
Therefore, \eqref{cond:theo2} and \eqref{u:eval2} ensure that

\begin{align*}
\displaystyle 
\sum_{i=1}^{n}a_i (C_{p_{i}+1}^j)^{2} \left( \left\|\hat{u}\right\|_{L^{p_{i}+1}(\mathring\Omega_{0}^j)}+C_{p_{i}+1}\rho \right)^{p_{i}-1}
&< 1 - \f{\lambda}{\lambda_1^j} \leq 1 - \f{\lambda}{\lambda_1(\Omega')},
\end{align*}
where $\lambda_1(\Omega')\geq \lambda_1^j$.
Because $C_{p_{i}+1}^j$ can be regarded as $ C_{p_{i}+1}(\Omega',\partial \Omega') $,
it follows from Lemma \ref{lem:positive} that $ u|_{\Omega'}\equiv 0 $.	

\subsubsection*{Case 2 --- when $ V(\mathring\Omega_0^j,\partial \mathring\Omega_0^j \backslash \Gamma_N) \neq H^1_0(\mathring\Omega_0^j) $}
The main difference from Theorem \ref{theo:main} is the possibility of this case (see $ \Omega_0^1 $ or $ \Omega_0^2 $ in Fig.~\ref{fig:domain0}).
Let $ \Omega' $ be an arbitrary subdomain of $ \mathring\Omega_0^j $.
To reach the desired fact (there exists no nodal domain of $ u $ inside $ \mathring\Omega^j_{0} $), it is necessary to prove that $ u|_{\Omega'} $ vanishes if it can be considered as a solution of  the D- or M-problem of \eqref{main:fpro} with the notational replacements $ \Omega \rightarrow \Omega' $, $ \Gamma_D \rightarrow \Gamma'_D $, and $ \Gamma_N \rightarrow \Gamma'_N $, where $ \Gamma'_N =\partial\Omega' \cap \Gamma_N$ (allowed to be empty) and $ \Gamma'_D=\partial \Omega' \backslash  \overline{\Gamma'_N}$.
When $ V(\Omega',\Gamma'_D)=H^1_0(\Omega') $, $ u $ can be considered as a solution of the D-problem on $ \Omega' $; therefore, the same argument as that in Case 1 is true.

We are left to consider the case in which $ u|_{\Omega'} $ is a solution of the M-problem where $ V(\Omega',\Gamma'_D) \neq H^1_0(\Omega') $.
Considering the zero extension outside $ \Omega' $ to $ \mathring\Omega_0^j $, the restriction $ u|_{\Omega'} $ can be regarded as a function in $ V(\mathring\Omega^j_0,\partial \mathring\Omega_0^j \backslash \Gamma_N) $;
note that $ u|_{\partial \Omega'} $ can be nonzero only on a subset of $ \Gamma_N $ (again, see $ \Omega_0^1 $ or $ \partial \Omega_0^2 $ in Fig.\ref{fig:domain0}).
Therefore, it follows that $\lambda_1(\Omega',\Gamma'_D)\geq \lambda_1(\mathring\Omega_0^j,\partial \mathring\Omega_0^j \backslash \Gamma_N)$ and 
$ C_{p+1}(\mathring\Omega_0^j,\partial  \mathring\Omega_0^j \backslash \Gamma_N) $
can be used as
$ C_{p+1}(\Omega',\Gamma'_D)$ for $p \in (1,p^*)$.
Thus, we make the same argument as that in Case 1 combined with Lemma \ref{lem:mix}.
\qed

\begin{rem}
	Section \ref{sec:constants} discusses explicit estimations for a lower bound of $ \lambda_1^j $ and upper bounds of $ C_{p_{i}+1}$ and $ C_{p_{i}+1}^j$.
\end{rem}

\begin{rem}\label{rem:th31weekend}
	We have assumed that $\Gamma_D$ and $\Gamma_N$ are connected sets to avoid redundant discussion.
	However, Theorem {\rm \ref{theo:mix}} remains true for many other cases, such as when $\Omega = (0,1)^2$ and $\Gamma_D=\{(x,y)\in \mathbb{R}^2 : y=0,~ 0<x<1\} \cup \{(x,y)\in \mathbb{R}^2 : y=1,~ 0<x<1\}$.
	Note that, in this case, a solution of \eqref{eq:poisson} with the mixed boundary condition has $H^2$-regularity for $h \in L^2(\Omega)$ (see \cite[Subsection 5.3]{azegami2020boundary}).
\end{rem}

\begin{figure}[h]
	\begin{center}
		\includegraphics[height=70 mm]{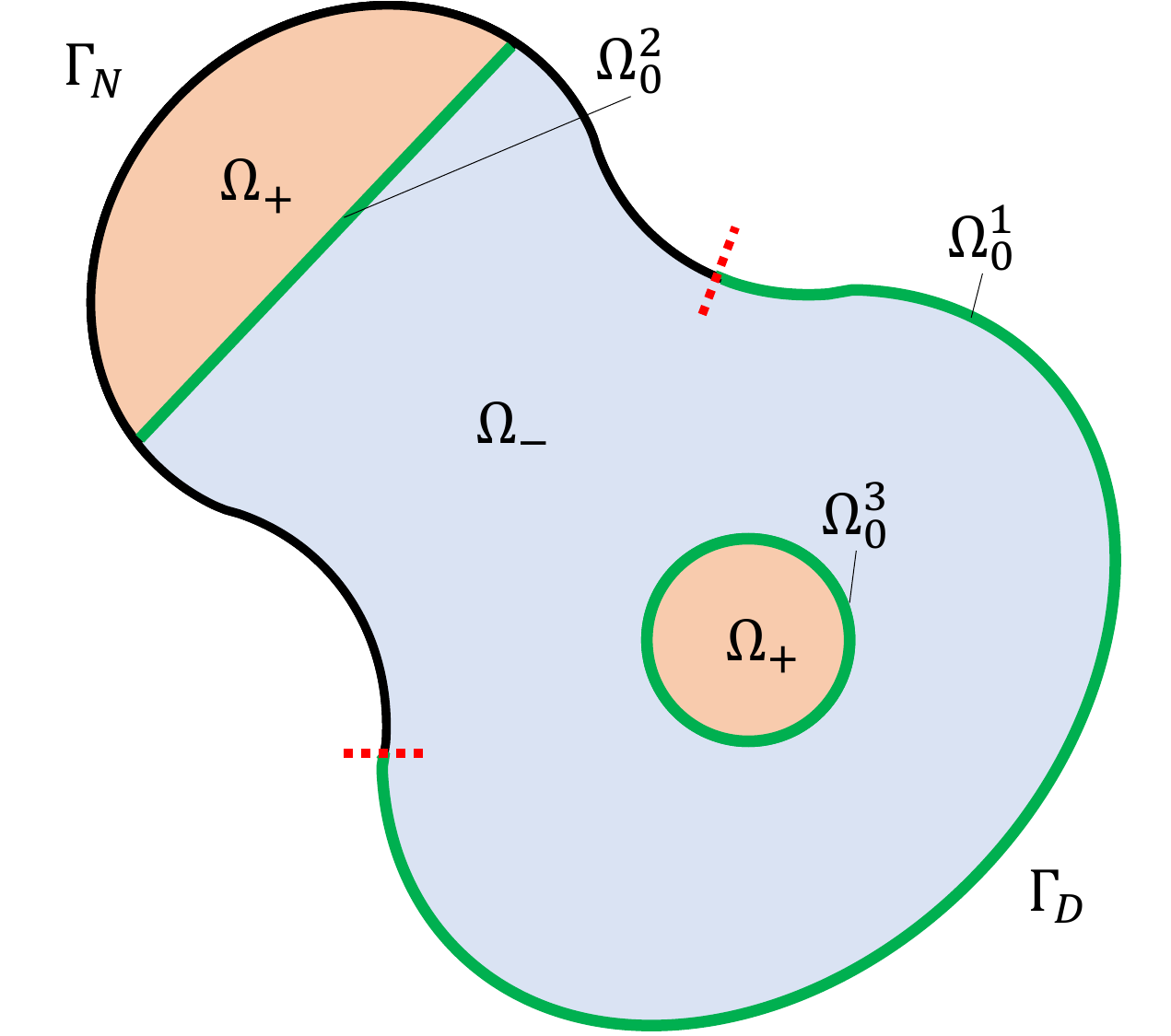}
	\end{center}
	\caption{Conceptual image of domains $\Omega$, $\Omega_{+}$, $\Omega_{-}$, and $\Omega_{0}$.
		The upper side (black line) is imposed on the Neumann boundary condition.
		The lower side (green line) is imposed on the Dirichlet boundary condition.
		The green lines satisfy $ \Omega_0 = \Omega_0^1 \cup \Omega_0^2 \cup \Omega_0^3$, which includes the part of $ \overline{\Omega} $ where $ |\hat{u}|\leq \sigma $ .
		These green lines are expected to topologically approximate the nodal lines of $ u $ in the sense that no nodal domain exists inside them.
		Note that $ \partial \Omega_0^3 $ consists of Dirichlet boundaries in their entirety, whereas some parts of $ \partial \Omega_0^1 $ and $ \partial \Omega_0^2 $ are Neumann boundaries located on their ``ends''.
	}
	\label{fig:domain0}
\end{figure}
\section{Conclusion}
We proposed a rigorous numerical method for analyzing the sign-change structure of solutions of the semilinear elliptic equation \eqref{eq:main}.
Given two types of error estimates $\left\|u-\hat{u}\right\|_{H^1_0}$ and $\left\|u-\hat{u}\right\|_{L^{\infty}}$ between an exact solution $ u $ and a numerically computed approximate solution $ \hat{u} $,
we provided a method for estimating the number of nodal domains (see Theorems \ref{theo:main} and \ref{theo:mix}).
The location of the nodal line of $ u $ can be determined via the information of $ \hat{u} $ and a verified $ L^{\infty} $-error $ \sigma $.
Our method was used to analyze the sign-change structure of the Allen--Cahn equation \eqref{eq:allen} subject to the homogeneous Dirichlet boundary condition.
In Section \ref{sec:othreboundary}, our method was extended to Neumann and mixed boundary conditions (see Theorem \ref{theo:mix}).
\appendix
\renewcommand{\thetheo}{\Alph{section}.\arabic{theo}}
\section{Required constants --- eigenvalues and embedding constants}\label{sec:constants}
In this section, we discuss evaluating the minimal eigenvalue $ \lambda_{1}(\mathring\Omega_0) $ and embedding constants $ C_{p+1} $ required in Theorems \ref{theo:main} and \ref{theo:mix}.

The following theorem can be used to obtain an explicit lower bound for the $ k $-th eigenvalue $ \lambda_{k}(\Omega) $ of the Laplacian imposed on the homogeneous Dirichlet boundary condition for a bounded domain $ \Omega $.
\begin{theo}[\cite{li1983schrodinger}] \label{theo:lower_eigen}
	Let $ \Omega \subset \mathbb{R}^{N}$~$(N=1,2,3,\cdots)$ be a bounded domain.
	We have
	\begin{align}
	\lambda_{k}(\Omega) \geq \frac{4\pi^2 N}{N+2} \left(\frac{k}{B_N|\Omega|}\right)^{\frac{2}{N}}, 
	\end{align}
	where $ |\Omega| $ and $ B_N $ denote the volume of $ \Omega $ and the unit $ N $-ball, respectively.
\end{theo}
Adapting Theorem \ref{theo:lower_eigen} to the case in which $ N=2,3 $, we have the following estimations for the first eigenvalue.
\begin{cor}\label{coro:lower_eigen}
	Under the same assumption of Theorem {\rm \ref{theo:lower_eigen}}, we have
	\begin{align*}
	&\lambda_{1}(\Omega) \geq 2 \pi |\Omega|^{-1},&N=2,\\
	&\lambda_{1}(\Omega) \geq \frac{3\times 6^{\frac{2}{3}}}{5} \pi^{\frac{4}{3}} |\Omega|^{-\frac{2}{3}},&N=3.
	\end{align*}	
\end{cor}

Theorem {\rm \ref{theo:lower_eigen}} $ ( $or Corollary {\rm \ref{coro:lower_eigen}}$ ) $ is reasonable for obtaining rough lower bounds for $ \lambda_1(\mathring\Omega_{0})~(=\lambda_1(\mathring\Omega_{0},\partial\mathring\Omega_{0})) $.
The Temple–Lehmann–Goerisch method can be helpful for us to obtain a more accurate evaluation of $\lambda_1(\mathring\Omega_{0})$
if more accuracy is required to satisfy the inequalities assumed in Theorems \ref{theo:main} and \ref{theo:mix} (see, for example, \cite[Theorem 10.31]{nakaoplumwatanabe2019numerical}).
Another possible approach  is
Liu's method provided in {\rm \cite{liu2013verified,liu2015framework}},
which is based on the finite element method and can be applied also to estimate the eigenvalues $ \lambda_1(\mathring\Omega_0^j,\partial \mathring\Omega_0^j \backslash \Gamma_N) $ (corresponding to a mixed boundary condition) required in Theorem \ref{theo:mix}.

Upper bounds for $ C_{p}(\Omega)\,(= C_{p}(\Omega, \partial\Omega))$ can be estimated via {\rm \cite[Corollary A.2]{tanaka2017sharp}} or \cite[Lemma 2]{plum2008} , which are used in the numerical examples in Subsection \ref{subsec:ex}.
Before introducing them, we cite the following famous result for the best constant in the classical Sobolev inequality.
Hereafter, the range of $ p $ is shifted by 1 (in place of $ p+1 $) to fit the original notation.
\begin{theo}[\cite{aubin1976} and \cite{talenti1976}]\label{theo:talenti}
	Let $u$ be any function in $W^{1,q}\left(\mathbb{R}^{N}\right)\ (N=2,3,\cdots)$.
	Moreover, let $q$ be any real number such that $1<q<N$, and let $p=Nq/\left(N-q\right)$.
	Then,
	\begin{align*}
	\left(\int_{\mathbb{R}^{N}}\left|u(x)\right|^{p}dx\right)^{\frac{1}{p}}\leq T_{p}\left(\int_{\mathbb{R}^{N}}\left|\nabla u(x)\right|_{2}^{q}dx\right)^{\frac{1}{q}}
	\end{align*}
	holds for
	\begin{align}
	T_{p}=\pi^{-\frac{1}{2}}N^{-\frac{1}{q}}\left(\frac{q-1}{N-q}\right)^{1-\frac{1}{q}}\left\{\frac{\Gamma\left(1+\frac{N}{2}\right)\Gamma\left(N\right)}{\Gamma\left(\frac{N}{q}\right)\Gamma\left(1+N-\frac{N}{q}\right)}\right\}^{\frac{1}{N}}\label{talenticonst},
	\end{align}
	where
	$\left|\nabla u\right|_{2}=\left((\partial u/\partial x_{1})^{2}+(\partial u/\partial x_{2})^{2}+\cdots+(\partial u/\partial x_{N})^{2}\right)^{1/2}$,
	and $\Gamma$ denotes the gamma function.
\end{theo}

The following corollary obtained from Theorem \ref{theo:talenti} provides a simple bound for the embedding constant from $H_{0}^{1}\left(\Omega\right)$ to $L^{p}(\Omega)$ for a bounded domain $\Omega$, where $H_{0}^{1}(\Omega)$ is endowed with the usual norm $ \|\nabla \cdot \|_{L^2} $.
Recall that this can be used as an upper bound for the embedding constant with the generalized norm \eqref{tau_norm}.
\begin{cor}[{\rm \cite[Corollary A.2]{tanaka2017sharp}} ]\label{coro:roughembedding}
	Let $\Omega\subset \mathbb{R}^{N} (N=2,3,\cdots)$ be a bounded domain.
	Let $p$ be a real number such that $p\in(N/(N-1),2N/(N-2)]$ if $N\geq 3$ and $p\in(2,\infty)$ if $N=2$.
	Additionally, set $q=Np/(N+p)$.
	Then, $(\ref{embedding})$ holds for
	\begin{align*}
	C_{p}\left(\Omega\right)=\left|\Omega\right|^{\frac{2-q}{2q}}T_{p},
	\end{align*}
	where $T_{p}$ is the constant in {\rm (\ref{talenticonst})}.
\end{cor}

The following theorem estimates the embedding constants, where $H_{0}^{1}(\Omega)$ is endowed with the generalized norm \eqref{tau_norm}.
This theorem is applicable to unbounded domains.
\begin{theo}[{\cite[Lemma 7.10]{nakaoplumwatanabe2019numerical}} ]
\label{theo:plum_embedding}
	Let $\Omega\subset \mathbb{R}^{N} (N=2,3,\cdots)$ be a bounded or unbounded domain.
	Let $\lambda_{1} \in[0,\infty)$ denote the minimal point of the spectrum of $-\Delta$ on $H_{0}^{1}(\Omega)$ endowed with the inner product \eqref{tau_norm},
	where $\tau$ is selected so that $\tau>0$ when $\lambda_{1}=0$.
	For $s \in [0,1]$, we define
	\begin{align*}
	\displaystyle
	\gamma_s :=
	\left\{\begin{array}{l l}
	\frac{s^s(1-s)^{1-s}}{\tau ^s} & \text{~if~} s \lambda_{1} \leq (1-s)\tau,\\[5pt]
	\frac{\lambda_1 ^{1-s}}{\lambda_1 + \tau} & \text{~otherwise},
	\end{array}\right.
	\end{align*}
	where $0^0:=1$.
	
	\noindent$a)$~~Let $N=2$ and $p\in[2,\infty).$
	With the largest integer $\nu$ satisfying $\nu\leq p/2, \ (\ref{embedding})$ holds for
	\begin{align*}
	C_{p}=\left(\frac{1}{2}\right)^{\frac{1}{2}+\frac{2\nu-3}{p}}\left[\frac{p}{2}\left(\frac{p}{2}-1\right)\cdots\left(\frac{p}{2}-\nu+2\right)\right]^{\frac{2}{p}} \sqrt{\gamma_{\frac{2}{p}}},
	\end{align*}
	where $\displaystyle \frac{p}{2}\left(\frac{p}{2}-1\right)\cdots\left(\frac{p}{2}-\nu+2\right)=1$ if $\nu=1.$\\[1pt]
	$b)$~~Let $N\geq 3$ and $p\in[2,2N/(N-2)]$.
	With $s:=N(p^{-1}-2^{-1}+N^{-1})\in[0,1],\ (\ref{embedding})$ holds for
	\begin{align*}
	C_{p}=\left(\frac{1}{\sqrt{N(N-2) \pi}}\left[\frac{\Gamma(N)}{\Gamma\left(\frac{N}{2}\right)}\right]^{\frac{1}{N}}\right)^{1-s} \sqrt{\gamma_{s}}.
	\end{align*}
\end{theo}

Although Corollary \ref{coro:roughembedding} and Theorem \ref{theo:plum_embedding} are reasonable for evaluating embedding constants under the homogeneous Dirichlet boundary conditions, Theorem \ref{theo:mix} requires upper bounds for more general constants $ C_{p_{i}+1}(\Omega,\partial \Omega \backslash \Gamma_N) $ and $ C_{p_{i}+1}(\mathring\Omega_0^j,\partial \mathring\Omega_0^j \backslash \Gamma_N) $.
Generally, directly evaluating the best values of these embedding constants is not easy.
Instead of a direct estimation, we can use the bound for embedding $ H^1(\Omega) \hookrightarrow L^{p+1}(\Omega) $ as an upper bound.
Such an upper bound is provided, for example, in {\rm \cite{tanaka2015estimation,mizuguchi2017estimation}} although estimations derived using these methods are rather larger than those in the homogeneous Dirichlet case.
Therefore, for Case 2 in the proof of Theorem \ref{theo:mix}, inequality \eqref{cond:theo2} is less likely to hold than for Case 1.
In the following, we introduce \cite[Theorems 2.1 and 3.3]{mizuguchi2017estimation} , which provide reasonable estimates for the embedding constant.
These can be applied to a domain $\Omega$ that can be divided into a finite number of bounded convex domains $\Omega_i~(i=1,2,3,\cdots, n)$ such that 
\begin{align}\label{omega1}
	\overline{\Omega}=\displaystyle\bigcup_{1\leq i\leq n}\overline{\Omega_i}
	\text{~~and~~}
	\Omega_i\cap\Omega_j=\emptyset~(i\neq j).
\end{align}
\begin{theo}\label{theo:mizuguchiDp}
	Let $\Omega\subset\mathbb{R}^N$ be a bounded convex domain.
	Moreover, let $ d_{\Omega}:=\sup_{x,y \in \Omega}|x-y| $, $\Omega_{x}:=\{x-y\,:\,y\in\Omega\}$ for $x\in\Omega$, and $U:=\cup_{x\in\Omega}\Omega_x$.
	Suppose that $1\leq q\leq p<qN/(N-q)$ if $N>q$, and $1\leq q\leq p<\infty$ if $N=q$. 
	Then, we have
	\begin{align}\label{result0}
		\|u-u_{\Omega}\|_{L^p(\Omega)}\leq D_p(\Omega)\|\nabla u\|_{L^q(\Omega)}~~~\text{for all}~u\in W^{1,q}(\Omega) 
	\end{align}
	with
	\begin{align*}
		D_p(\Omega)=\frac{d_\Omega^{N}}{N|\Omega|}(A_{r}A_{q}A_{p'})^N \||x|^{1-N}\|_{L^{r}(U)},
	\end{align*}
	where $u_{\Omega}=|\Omega|^{-1} \int_{\Omega} u(x) dx$, $r=qp/((q-1)p+q)$, and
	\begin{align*}
		A_m=
		\begin{cases}
		\sqrt{m^{\frac{2}{m}-1}(m-1)^{1-\frac{1}{m}}}~~(1<m<\infty),\\[2mm]
		\hspace{1.5cm}1\hspace{1.8cm}(m=1,~\infty).
		\end{cases}
	\end{align*} 
\end{theo}
\begin{theo}\label{theo:mizuguchCp} 
	Let $\Omega\subset\mathbb{R}^N$ be a bounded domain,
	and let $p$ and $q$ satisfy $1\leq q\leq p\leq\infty$.
	Suppose that there exists a finite number of bounded domains $\Omega_i~(i=1,2,3,\cdots, n)$ satisfying \eqref{omega1}.
	Moreover, suppose that for every $\Omega_i~(i=1,2,3,\cdots, n)$ there exist constants $D_p(\Omega_i)$ such that
	\begin{align}\label{mainine}
	\|u-u_{\Omega_i}\|_{L^{p}(\Omega_i)}\leq D_{p}(\Omega_i)\|\nabla u\|_{L^q(\Omega_i)} ~~~\text{for all}~u\in W^{1,q}(\Omega_i).
	\end{align}
	Then, 
	\begin{align}\label{aim}
	\left(\int_{\Omega}|u(x)|^pdx\right)^{\frac{1}{p}}\leq C_p'(\Omega)\left(\int_{\Omega}|u(x)|^qdx+\int_{\Omega}|\nabla u(x)|^qdx\right)^{\frac{1}{q}}
	\end{align}	
	holds for
	\begin{align}
	C_p'(\Omega)=\begin{cases}
	\max\left(1,~\displaystyle\max_{1\leq i\leq n}D_{\infty}(\Omega_i)\right) & (p=q=\infty)\\[2mm]
	2^{1-\frac{1}{q}}\max\left(\displaystyle\max_{1\leq i\leq n}|\Omega_i|^{\frac{1}{p}-\frac{1}{q}},~\displaystyle\max_{1\leq i\leq n}D_p(\Omega_i)\right) & (otherwise),
	\end{cases}
	\label{sec3:theo:cp}
	\end{align}
	where this formula is understood with $1/\infty=0$ when $p=\infty$ and/or $q=\infty$.
\end{theo}

Using Theorems \ref{theo:mizuguchiDp} and \ref{theo:mizuguchCp} with $ q=2 $, we can estimate required bounds for embedding constants.
Indeed, $ C_p'(\Omega) $ in \eqref{aim} with $ q=2 $ becomes an upper bound for $ C_{p}(\Omega,\Gamma_D) $ for any choices of $ \Gamma_D $.

\begin{acknowledgements}
	We express our sincere thanks to Prof. Kazunaga Tanaka (Waseda University, Japan) for his helpful advice, Prof. Mitsuhiro T.~Nakao (Kyushu University / Waseda University, Japan) for insightful comments, and Taisei Asai (Waseda University, Japan) for helping to create easy-to-read figures and tables.
	We also express our profound gratitude to an anonymous referee for highly insightful comments and suggestions.
\end{acknowledgements}

%
%

\bibliographystyle{spmpsci}      
\bibliography{ref}   


\end{document}